\documentclass{amsproc}

\usepackage[pdftex]{hyperref}
\usepackage{color}

\usepackage{amsmath,amssymb,amsthm,mathtools}
\usepackage[margin=1in]{geometry}
\usepackage[only,llbracket,rrbracket]{stmaryrd}
\usepackage[numbers,square]{natbib}

\newtheorem{theorem}{Theorem}[section]
\newtheorem{lemma}[theorem]{Lemma}

\newtheorem{proposition}[theorem]{Proposition}
\newtheorem{conjecture}[theorem]{Conjecture}

\theoremstyle{definition}
\newtheorem{definition}[theorem]{Definition}
\newtheorem{example}[theorem]{Example}

\theoremstyle{remark}
\newtheorem{remark}[theorem]{Remark}

\numberwithin{equation}{section}

\newcommand{\R}{\mathbb{R}}
\newcommand{\N}{\mathbb{N}}
\newcommand{\C}{\mathbb{C}}
\newcommand{\Z}{\mathbb{Z}}

\newcommand{\dd}{\mathrm{d}}
\newcommand{\sg}{\sigma}
\newcommand{\Pn}{P_n}

\newcommand{\He}{\mathrm{He}}
\newcommand{\ii}{{\rm{i}}}
\newcommand{\eee}{{\rm e}}

\renewcommand{\Re}{\operatorname{Re}}
\renewcommand{\Im}{\operatorname{Im}}

\begin{document}

\title[Multiplicative Hermite and Laguerre polynomials]{Zero distribution of multiplicative Hermite and Laguerre polynomials}


%    author one information
\author{Zakhar Kabluchko}
\address{Institut f\"ur Mathematische Stochastik, Universit\"at M\"unster, Orl\'eans-Ring 10, 48149 M\"unster, Germany}
\email{zakhar.kabluchko@uni-muenster.de}

\thanks{This work would not be possible without the selfless help and clever insights of  ChatGPT5. Supported by the German Research Foundation under Germany's Excellence Strategy  EXC 2044 -- 390685587, Mathematics M\"unster: Dynamics - Geometry - Structure and by the DFG priority program SPP 2265 Random Geometric Systems.
}

%    The 2020 edition of the Mathematics Subject Classification is
%    the current definitive version.
 %% Checked


\date{}

\begin{abstract}
It is well-known that, as $n\to\infty$, the zero distribution of the $n$-th Hermite polynomial converges to the semicircular law (the free normal distribution), while the zero distribution  of the associated Laguerre polynomials converges to the Marchenko--Pastur law (the free Poisson distribution). In this paper, we establish multiplicative analogues of these results. We define the multiplicative  Hermite and Laguerre polynomials by
\begin{align*}
H_n^*(x;s)
&:=
\eee^{-\frac s2 ((x\partial_x)^2 - n x \partial_x) } (x-1)^n  = \sum_{j=0}^n (-1)^{n-j} \binom nj \eee^{-\frac s2 (j^2 - nj)} x^j,
\\
L_n^*(x; b,c)
&:=
(x\partial_x + b)^c (x-1)^n =  \sum_{j=0}^n (-1)^{n-j} \binom nj (j+b)^c x^j,
\end{align*}
where $n\in \N_0$, $\partial_x$ denotes the differentiation operator w.r.t.\  $x$,  and $s\in \R$, $b\in \C$, $c\in \N_0$ are parameters. For the multiplicative Hermite polynomials, we show that, as $n\to\infty$,  the zero distribution of $H_n^*(x;s/n)$ converges weakly to the free multiplicative normal distribution on the positive half-line (when $s>0$) or to the free unitary normal distribution on the unit circle $\{|z| = 1\}$ (when $s<0$). For the multiplicative Laguerre polynomials,  we show that the zero distribution of $L_n^*(x; n\beta, \lfloor n \gamma \rfloor)$ converges to the free multiplicative Poisson distribution on the positive half-line (when $\gamma >0$ and $\beta \in \R\backslash[0,1]$) or on the unit circle (when $\gamma>0$ and  $\beta \in -\frac 12 + \ii \R$). All these results are obtained by essentially the same method, which treats the Hermite/Laguerre cases and the unitary/positive settings in a unified way. As a corollary, let $(P_n(x))_{n\in \N}$ be a sequence of polynomials with $\deg P_n = n$ whose zeros are all nonnegative or all unitary, and assume that the zero distribution  of $P_n(x)$ converges to a probability measure $\rho$, as $n\to\infty$. Then the asymptotic zero distributions  of the polynomials  $\exp(- \frac s{2n} ((x \partial_x)^2- n x \partial_x)) P_n(x)$ and   $(x \partial_x + n \beta)^{\lfloor n \gamma\rfloor}  P_n(x)$ are  characterized, respectively, as the free multiplicative convolutions of $\rho$ with the free multiplicative normal and free multiplicative Poisson distributions.
\end{abstract}

\maketitle

\section{Hermite polynomials, heat flow, and roots of polynomials}
\subsection{Introduction}
It is well-known that, as $n\to\infty$, the zero distribution of the $n$-th Hermite polynomial converges to the semicircular law. Similarly, the zero distribution of the (associated) Laguerre polynomial converges to the Marchenko--Pastur distribution. In free probability, these two distributions are analogues of the normal and the Poisson distribution, respectively. The finite free probability~\cite{marcus2021polynomial,marcus_spielman_srivastava} completes the picture by showing that, for a fixed degree $n$,  the Hermite and the Laguerre polynomials appear as limits in the finite free analogues of the central limit theorem and the Poisson limit theorem.

In the present paper, we study the asymptotic zero distribution of  the \emph{multiplicative} analogues of the Hermite and Laguerre polynomials. We show that the zero distribution of these polynomials converges, as the degree goes to $\infty$, to the free multiplicative normal (respectively, Poisson) distribution. Any such result exists in two settings, a unitary (on the unit circle) and a positive (on the positive half-line). We give proofs that treat these two settings in a unified way.

\subsection{Hermite polynomials and the heat flow}\label{subsec:hermite_polys_additive}
We begin by recalling some well-known facts about Hermite polynomials and  their relationship to the heat operator. The classical (probabilists') Hermite polynomials  $\He_1(x;t)=x$, $\He_2(x;t)=x^2-t$, $\He_3(x;t)=x^3-3tx, \ldots$  may be defined by the formula
\begin{equation}\label{eq:hermite_polys_def}
\He_n(x;t)
=
\eee^{-\frac12 t \partial_x^2} x^n
=
\sum_{m=0}^{\lfloor n/2\rfloor}\frac{n!}{m! (n-2m)!}\left(-\frac{t}{2}\right)^{m}x^{ n-2m},
\qquad n\in \N_0, \; t\in \R.
\end{equation}
Here, $\partial_x$ denotes the differentiation operator w.r.t.\  the variable $x$, and $\eee^{\frac12 s \partial_x^2}: \C[x] \to \C[x]$ denotes the heat operator with ``time'' $s\in \C$, a linear operator acting on the vector space $\C[x]$ of polynomials via
$$
\eee^{\frac12 s \partial_x^2} P(x): = \sum_{j=0}^\infty\frac{(s/2)^j}{j!} \partial_x^{2j}P(x), \qquad P(x) \in \C[x].
$$
Note that the series on the right-hand side terminates after finitely many non-zero terms. The variable $t$ in $\He_n(x;t)$ is a  scaling parameter: Defining $\He_n(x):= \He_n(x; 1)$ we have
\begin{equation}\label{eq:hermite_scaling}
\He_n(x;t)=t^{n/2} \mathrm{He}_n\left(\frac{x}{\sqrt{t}}\right) \quad(\text{if } t>0),
\qquad
\He_n(x;t) = (\ii\sqrt{-t})^{ n} \mathrm{He}_n\left(\frac{x}{\ii\sqrt{-t}}\right)  \quad(\text{if } t<0).
\end{equation}
If $P(x)\in \C[x]$ is a polynomial  playing the role of the initial condition, then the solution to the heat equation
\[
\partial_s u(x; s)  = \frac{1}{2}\partial_x^2 u(x;s), \qquad u(x;0)=P(x),\qquad x\in \R, \;   s\in \R,
\]
can be written as
$$
u(x;s) =\eee^{\frac{1}{2}s\partial_x^2} P(x), \qquad s\in \R.
$$
The operator $\eee^{\frac{1}{2}s\partial_x^2}$ is called the forward heat operator if $s>0$ and the backward heat operator if $s<0$.

\subsection{Heat flow and roots of polynomials}
An interesting property of the \emph{backward} heat flow is that it preserves real-rootedness: If $P(x)$ is a polynomial with only real roots, then the polynomial  $\eee^{-\frac{t}{2}\partial_x^2} P(x)$ is also real-rooted  for all $t\geq 0$. This claim is a special case of the P\'olya--Benz theorem; see~\cite{benz} and~\cite[Theorem~1.2]{aleman_beliaev_hedenmalm}. For another proof, see~\cite{tao_blog1}.  For example, it is a classical fact that $\He_n(x;t) = \eee^{-\frac12 t \partial_x^2} x^n$ is real-rooted for all $t>0$. On the other hand, it follows from~\eqref{eq:hermite_scaling} that the \emph{forward} heat flow preserves the class of polynomials whose roots belong to the imaginary axis $\ii   \R$.

The \textit{zero distribution} of a polynomial $P(x)\in \C[x]$ of degree  $n\in \N$ is a probability measure on $\C$ given by
\begin{equation*}
\llbracket  P \rrbracket_n  \coloneqq \frac 1n \sum_{\substack{z\in \C:  P(z) = 0}} \delta_z,
\end{equation*}
where the sum is taken over all complex zeros of $P$ counting multiplicities. Here, $\delta_z$ is the unit mass at $z$. The following fact is classical, see, e.g., \cite[Theorem~3]{gawronski_on_the_asymptotic} or~\cite[Theorem~3.2(b)]{dette_studden}.
\begin{theorem}[Asymptotic zero distribution of Hermite polynomials]\label{theo:zeros_hermite_classical_asymptotic_distr}
For $t>0$, the probability measures $\llbracket  \He_n(\cdot; t/n)\rrbracket_n$ converge weakly (as $n\to \infty$) to the semicircle distribution $\mathsf{sc}_t$ on the interval $[-2\sqrt t,2 \sqrt t]$ with Lebesgue density $x\mapsto \sqrt{4t - x^2}/(2\pi t)$.
\end{theorem}

\begin{remark}
For $t<0$, \eqref{eq:hermite_scaling} entails that the zeros of $\He_n(\cdot; t/n)$ are purely imaginary and $\llbracket  \He_n(\cdot; t/n)\rrbracket_n$ converges weakly to a probability measure on $\ii   \R$ which is the image of $\mathsf{sc}_{-t}$ under the map $z \mapsto \ii z$.
\end{remark}

There has been some recent interest on how differential operators (such as repeated differentiation~\cite{bogvad_etal,campbell_etal_R_diagonal,HHJK_repeated_diff,hoskins_kabluchko,jalowy_kabluchko_marynych_zeros_profiles_part_I,jalowy_kabluchko_marynych_zeros_profiles_part_II,kiselev_tan_flow_repeated,martinez_finkelshtein_rakhmanov,Steiner21} or the heat flow~\cite{HHJK_heatflow,hoefert_kabluchko_jalowy,kabluchko2024leeyangzeroescurieweissferromagnet,tao_blog1,tao_blog2}) affect the asymptotic zero distribution of a high-degree polynomial. The subject has connections to free probability, PDE's, random matrices, Calogero--Moser systems  and (via the de Bruijn--Newman constant) the Riemann hypothesis. For the last point, see~\cite{newman_wu,polymath_heat_flow_riemann_zeta,rodgers_tao,tao_blog1,tao_blog2}.
The next theorem describes how the backward heat flow affects the asymptotic zero distribution of a high-degree real-rooted polynomial. It can be found in~\cite[Theorem~2.13]{kabluchko2024leeyangzeroescurieweissferromagnet}. (An equivalent result, stated in the language of Calogero--Moser ODE's can be found in~\cite[Theorem~1.1]{voit_woerner}.)

\begin{theorem}[Backward heat flow and the asymptotic distribution of zeros]\label{theo:heat_flow_zeroes_algebraic}
Let $(P_n(x))_{n\in \N}$ be a sequence of polynomials  such that $\deg P_n = n$ for all $n\in \N$. Fix $t>0$.  Suppose that all zeros of each $P_n$ are real and belong to some  interval $[-C,C]$. Suppose also that, as $n\to\infty$,  $\llbracket  P_n \rrbracket_n$ converges  to some probability measure  $\rho$ weakly on $\R$. Then,  for every $t\geq 0$,  $\llbracket  \eee^{- \frac t{2n} \partial_x^2 } P_n\rrbracket_n$ converges  to $\rho \boxplus \mathsf{sc}_t$ weakly on $\R$. Here, $\boxplus$ denotes the free additive convolution of probability measures on $\R$.
\end{theorem}

\begin{example}
Taking $P_n(x) = x^n$, which has $\rho = \delta_0$, and recalling that $\He_n(x;t/n) = \eee^{- \frac t{2n} \partial_x^2 } x^n$, we recover  Theorem~\ref{theo:zeros_hermite_classical_asymptotic_distr}.
\end{example}

\section{Main results on multiplicative Hermite polynomials}
\subsection{Definition of multiplicative Hermite polynomials}
In the present paper, we shall be interested in the multiplicative analogues of the results described above. To pass from the ``additive'' setting described in Section~\ref{subsec:hermite_polys_additive} to the multiplicative one, we use the exponentiation map $x\mapsto \eee^x$, which is an isomorphism between the additive group $(\R, +)$ and the multiplicative group $(\R_{>0}, \cdot)$. The image of the Lebesgue measure $\dd x$ on $\R$ under the exponentiation map is the measure $\dd x/x$ on $\R_{>0}$. The map $U: L^2 (\R_{>0}, \dd x /x) \to L^2(\R, \dd x)$ given by $(U f)(x) = f(\eee^x)$ is an isomorphism of Hilbert spaces and intertwines the operator $\partial_x$ acting on smooth functions in $L^2(\R, \dd x)$ with the operator $x\partial_x$ acting on smooth functions in $L^2 (\R_{>0}, \dd x/x)$:
$$
U  \cdot x\partial_x  = \partial_x \cdot  U
$$

The multiplicative analogue of the heat operator is the operator $\eee^{\frac{s}{2}(x\partial_x)^2}: \C[x] \to \C[x]$, where $s\in \C$ is the ``time''. Since $x\partial_x x^k = k x^k$, $k\in \N_0$,  it is natural to define
$$
\eee^{\frac{s}{2}(x\partial_x)^2} x^k = \eee^{\frac s2 k^2} x^k,  \qquad k\in \N_0.
$$
If $P(x)\in \R[x]$ is a polynomial, then the solution of the multiplicative heat equation
\[
\partial_s v(x; s)  = \frac{1}{2} (x\partial_x)^2 v(x;s), \qquad v(x;0)=P(x), \qquad x>0, \; s\in \R,
\]
can be written as
$$
v(x;s) =\eee^{\frac{s}{2}(x\partial_x)^2} P(x), \qquad s\in \R.
$$

The multiplicative analogue of the polynomial $x^n$, for which we have $\llbracket  x^n \rrbracket_n= \delta_0$, is the polynomial $(x-1)^n$, which satisfies  $\llbracket  (x-1)^n \rrbracket_n= \delta_1$. As a multiplicative  analogue of~\eqref{eq:hermite_polys_def}, we consider the following

\begin{definition}
The \emph{multiplicative Hermite polynomials} with parameter $s\in \C$ are defined by
\begin{equation}\label{eq:hermite_multipl_def}
H_n^*(x;s) := \eee^{-\frac s2 ((x\partial_x)^2 - n x \partial_x) } (x-1)^n  = \sum_{j=0}^n (-1)^{n-j} \binom nj \eee^{-\frac s2 (j^2 - nj)} x^j, \qquad n\in \N_0.
\end{equation}
\end{definition}
Observe that, unlike in the classical Hermite case, $s$ is not just a scaling parameter.

\begin{remark}\label{rem:hermite_mult_symmetry}
The terms $-nx \partial_x$  and $-nj$ appearing in~\eqref{eq:hermite_multipl_def} were introduced for the following reason.  With these terms,  \eqref{eq:hermite_multipl_def}  implies that $z^n H_n^*(1/z;s) = (-1)^n H_n^*(z;s)$. In particular, the zeros of $H_n^*(\cdot; s)$ are invariant w.r.t.\ the map $x\mapsto 1/x$. This symmetry around $1$ is the multiplicative analogue of the property $\He_n(-z;s) = (-1)^n \He_n(z;s)$ which implies that the zeros of $\He_n$ are symmetric w.r.t.\ $0$.
\end{remark}
\begin{remark}
The polynomials $H_n^*(x;s)$ appeared in the work of Mirabelli~\cite[Section~3.2]{mirabelli2021hermitian} as limits in  the finite free multiplicative analogue of the central limit theorem. Similarly, the classical Hermite polynomials are the limits in the finite free additive analogue of the central limit theorem, see~\cite{marcus2021polynomial}.   In~\cite{kabluchko2024leeyangzeroescurieweissferromagnet}, the expected characteristic polynomial of the Brownian motion on the unitary group $U(n)$ has been expressed in terms of $H_n^*(x;s)$ with $s\leq 0$. Tao~\cite{tao_blog2} discussed these polynomials in connection with  the finite-field version of the de Bruijn--Newman constant. A Calogero--Moser-type ODE system satisfied by the roots of $H_n^*(x;s)$ can be found in~\cite{tao_blog2} or~\cite[Section~3.2.5]{mirabelli2021hermitian}.
\end{remark}

\subsection{Zeros of multiplicative Hermite polynomials}
Recall that $\He_n(x;t)$ is real-rooted for $t\geq 0$, while for $t\leq 0$, all zeros of $\He_n(x;t)$ belong to the imaginary axis $\ii   \R$. For $t=0$, both claims are true since $\He_n(x;0) = x^n$. The next proposition is a multiplicative version of this result.
\begin{proposition}[Positive and unitary zeros]\label{prop:mult_hermite_zeros_positive_unitary}
Let $n\in \N$.
\begin{itemize}
\item[(a)]
The positive case: For $s\geq 0$, all zeros of $H_n^*(\cdot; s)$ are real and positive.
\item[(b)]
The unitary case: For $s\leq 0$, all zeros of $H_n^*(\cdot; s)$ are located on the unit circle $\{z\in \mathbb C: |z| =1\}$.  \item[(c)] For all $s\in \C$, the multiset of zeros of $H_n^*(\cdot; s)$ is invariant w.r.t.\ the map $z\mapsto 1/z$.
\end{itemize}
\end{proposition}
\begin{proof}
For the case $s<0$, see Lemma~2.1 in~\cite{kabluchko2024leeyangzeroescurieweissferromagnet}. For the case $s>0$, see~\cite[Theorem~5.3]{jalowy_kabluchko_marynych_zeros_profiles_part_II}. For $s=0$, we have $H_n^*(x;0) = (x-1)^n$ and all zeros are equal to $1$. Part~(c) has been shown in Remark~\ref{rem:hermite_mult_symmetry}.
\end{proof}

Let $z_{1;n}(s)\in \C,\dots,z_{n;n}(s)\in \C$ be the zeros of $H_n^*(\cdot;s/n)$ (counted with multiplicity and listed in some order). Note the $1/n$-scaling in the ``time'' variable.   The zero distribution of $H_n^*(\cdot;s/n)$ is the probability measure
$$
\llbracket  H_n^*(\cdot; s/n)\rrbracket_n := \frac 1n \sum_{j=1}^n  \delta_{z_{j;n}(s)}.
$$
In the next theorem, we identify the weak limit of $\llbracket  H_n^*(\cdot; s/n)\rrbracket_n$ as $n\to\infty$.

\begin{theorem}[Asymptotic zero distribution of multiplicative Hermite polynomials]\label{theo:multiplicative_hermite_polys_zeros}
Fix some $s\in \mathbb R$. As $n\to\infty$, the probability measure $\llbracket  H_n^*(\cdot; s/n)\rrbracket_n$ converges weakly to $\mu^{(s)}$,  a probability measure on $\C$ uniquely characterized by the following properties:
\begin{itemize}
\item[(a)] The positive case: For $s\geq 0$, $\mu^{(s)}$ is supported on a compact subset of $(0,\infty)$.
\item[(b)] The unitary case: For $s\leq 0$, $\mu^{(s)}$ is supported on the unit circle $\{|z| = 1\}$.
\item[(c)] For every $s\in \R$, the measure $\mu^{(s)}$ is invariant with respect to the map $z\mapsto 1/z$.
\item[(d)] For every $s\in \R$, the moments of $\mu^{(s)}$ are given by
\begin{align}
\int_\C x^k \mu^{(s)}(\dd x)
&=
\frac{\eee^{sk/2}}{k} [u^{ k-1}] \left((1+u)^{k} \eee^{ sk u}\right) \label{eq:mu_s_moments1}
\\
&=
\frac{\eee^{sk/2}}{k}\sum_{j=0}^{k-1}\binom{k}{j+1} \frac{(s k)^{ j}}{j!} \label{eq:mu_s_moments2}
\\
&=
\frac{\eee^{sk/2}}{k} L_{k-1}^{(1)}(-ks),
\qquad
k\in \mathbb N, \label{eq:mu_s_moments3}
\end{align}
where $[u^{m}] f(u)$ is the coefficient of $u^m$ in the Taylor series of $f(u)$ and $L_n^{(\alpha)}(x) = \sum_{j=0}^{n}\binom{n+\alpha}{n-j}\frac{(-x)^{j}}{j!}$ are the associated Laguerre polynomials.
\end{itemize}
\end{theorem}

\begin{figure}[t]
  \centering
  \includegraphics[width=0.5\textwidth]{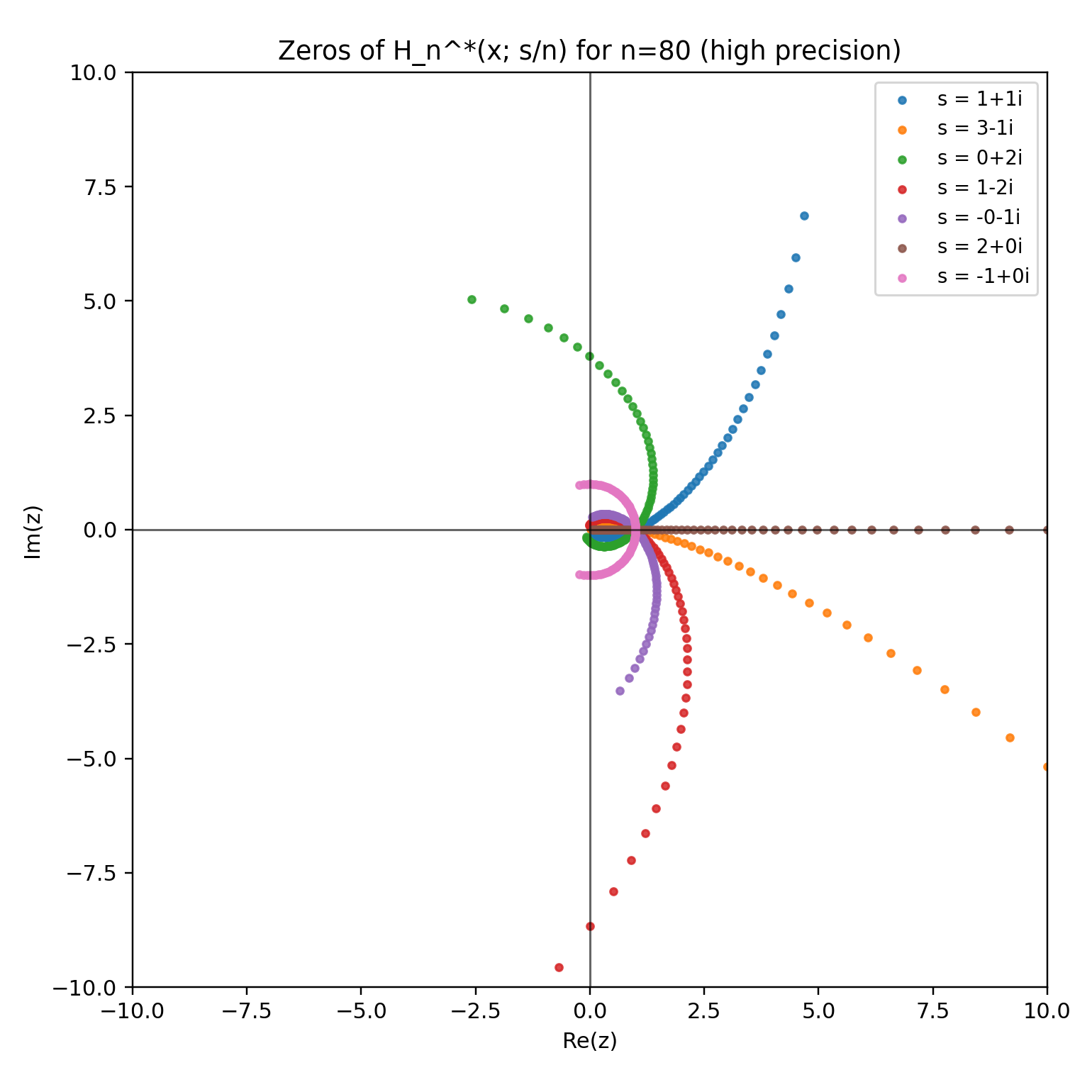}
  \caption{Zeros of $H_n^*(x;s/n)$ for $n=80$ and
  $s\in\{1+\ii, 3-\ii, 2\ii, 1-2\ii, -\ii, 2,-1\}$, plotted in the window $[-10,10]\times[-10,10]$.}
  \label{fig:zeros-overlay-n80-seven-s}
\end{figure}

For $s<0$, Theorem~\ref{theo:multiplicative_hermite_polys_zeros} was proved in~\cite{kabluchko2024leeyangzeroescurieweissferromagnet} using the saddle-point method. For $s>0$, it was proved in~\cite{jalowy_kabluchko_marynych_zeros_profiles_part_II} using exponential profiles (see~\cite{jalowy_kabluchko_marynych_zeros_profiles_part_I}). In Section~\ref{sec:proof_hermite} of the present paper, we prove Theorem~\ref{theo:multiplicative_hermite_polys_zeros} by a new method which has the advantage that it uniformly applies to both cases, $s<0$ and $s>0$. After minimal modifications, the method establishes convergence of analytic moments of $\llbracket  H_n^*(\cdot; s/n)\rrbracket_n$ for every \emph{complex} $s$. If $s$ is real, the zeros are positive/unitary, and weak convergence can be deduced from the convergence of analytic moments. For general complex $s$, numerical simulations, see Figure~\ref{fig:zeros-overlay-n80-seven-s}, suggest the following

\begin{conjecture}
For every complex $s\in \C$, the probability measures $\llbracket  H_n^*(\cdot; s/n)\rrbracket_n$ converge to some limit probability measure $\mu^{(s)}$ on $\C$ whose analytic moments are given by~\eqref{eq:mu_s_moments1}, \eqref{eq:mu_s_moments2}, \eqref{eq:mu_s_moments3}.
\end{conjecture}

\subsection{Free multiplicative normal distributions}
The probability measure $\mu^{(s)}$ is the free multiplicative normal distribution (unitary for $s\leq 0$ and positive for $s\geq 0$). For $s=0$, $\mu^{(0)}$ is the unit mass at $1$. These distributions appeared in the work of Bercovici and Voiculescu~\cite[Lemmas~6.3, 7.1]{bercovici_voiculescu_levy_hincin}, were extensively studied by Biane~\cite{biane,biane_segal_bargmann} and later in~\cite{Cebron2016AOP,Demni2017JOT,demni2016spectraldistributionlargesizebrownian,demni2011spectraldistributionfreeunitary,driver_hall_kemp,DriverHallKemp2013SB,zhong_free_brownian,zhong_free_normal}.
In free stochastic calculus~\cite{Biane1997FIC,Biane2001AIHP,BianeSpeicher1998PTRF}, the multiplicative normal distributions appear as the marginal distributions of the free multiplicative Brownian motions (unitary or positive).

In the next proposition, we list some properties of $\mu^{(s)}$ which can be found in these papers,
but first we need to recall some notions from free probability theory. For a compactly supported probability measure $\mu$ on $\mathbb C$ with moments $m_k=\int_{\C} x^k \mu(\dd x)$, $k\in \mathbb N_0$, the moment series $M_\mu(t)$, the $\psi$-transform $\psi_\mu(t)$,  and the Cauchy transform $G_\mu(z)$ are defined by
\begin{equation}\label{eq:M_G_transform}
M_\mu(t)=\sum_{k=0}^\infty m_k t^k,
\qquad
\psi_\mu(t)=M_\mu(t)-1,
\qquad
G_\mu(z)
=
\int_{\mathbb R}\frac{\mu(\dd x)}{z-x} =\frac{1}{z} M_\mu\left(\frac{1}{z}\right).
\end{equation}
The $R$-transform can be defined by either of the following formulas:
\begin{equation}\label{eq:R_transform}
G_\mu^{-1}(w)=\frac{1}{w}+R_\mu(w),
\qquad
M_\mu(t)=1+t M_\mu(t) R_\mu\left(t M_\mu(t)\right).
\end{equation}
The free cumulants $\kappa_1,\kappa_2,\ldots$ are defined by  $R_\mu(w)=\sum_{k\ge1}\kappa_k w^{k-1}$. The $S$-transform of $\mu$ is defined by
\begin{equation}\label{eq:S_transform_def}
S_\mu\left(\psi_\mu(t)\right)=\frac{1+\psi_\mu(t)}{\psi_\mu(t)} t
  =\frac{M_\mu(t)}{M_\mu(t)-1} t.
\end{equation}
The $R$-transform and the $S$-transform are analytic functions in a sufficiently small disk around $0$. It follows from~\eqref{eq:M_G_transform} and~\eqref{eq:R_transform} that $S_\mu(z) R_\mu(zS_\mu(z)) = 1$, locally around $0$. The next proposition collects some known properties of $\mu^{(s)}$.

\begin{proposition}[Free multiplicative normal distribution]
For all $s\in \R$, there exists  a unique probability measure $\mu^{(s)}$ on $(0,\infty)$ (when $s\geq 0$) or on the unit circle (when $s\leq 0$) whose moments are given by~\eqref{eq:mu_s_moments1}, \eqref{eq:mu_s_moments2}, \eqref{eq:mu_s_moments3}. The probability measure $\mu^{(s)}$ is compactly supported.  The $S$- and the $R$-transforms of $\mu^{(s)}$ are given by
$$
S(z;s)=\eee^{-s\left(z + \frac 12\right)},
\qquad
R(w;s)
=
\frac{W_0 \left(-s \eee^{s/2} w\right)}{-sw},
$$
where $W_0$ is the principal branch of the Lambert function solving $W_0(z) \eee^{W_0(z)} = z$.
The free cumulants of $\mu^{(s)}$ are given by
$$
\kappa_k = \frac{(s k)^{ k-1}}{k!} \eee^{sk/2} ,\quad k\in \N.
$$
\end{proposition}

\subsection{Multiplicative heat flow and the asymptotic distribution of zeros}
Recall that the operator $\eee^{-\frac{t}{2}\partial_x^2}$ preserves the class of real-rooted polynomials when $t\geq 0$ and the class of polynomials with zeros on $\ii   \R$ when $t\leq 0$. The next proposition is a multiplicative version of this result.
\begin{proposition}[Multiplicative heat flow preserves positive/unitary zeros]\label{prop:mult_heat_acts_in_zeros}
Let $P_n(x)\in \C[x]$ be a polynomial of degree $n\in \N$.
\begin{enumerate}
\item[(a)]  The positive case: If  all roots of $P_n$ are in $[0,\infty)$, then for every $s\geq 0$ all roots of the polynomials $\eee^{- \frac s{2} ((x \partial_x)^2- n x \partial_x) } P_n$ are also in $[0,\infty)$.
\item[(b)] The unitary case: If  all roots of $P_n$ are located  on the unit circle $\{|z|=1\}$, then for every $s\leq 0$ all roots of the polynomials $\eee^{- \frac s{2} ((x \partial_x)^2- n x \partial_x) } P_n$ are located on the unit circle.
\end{enumerate}
\end{proposition}
For the proof we need the multiplicative finite free convolution~\cite{marcus2021polynomial,marcus_spielman_srivastava} defined, for two polynomials of degree at most $n$, by
\begin{align}\label{eq:mult_finite_free_conv_def}
\sum_{j=0}^n a_{j;n}^{(1)} x^j \boxtimes_n \sum_{j=0}^n a_{j;n}^{(2)} x^j &=  \sum_{j=0}^n (-1)^{n-j} \frac{a_{j;n}^{(1)} a_{j;n}^{(2)}}{\binom n j} x^j.
\end{align}
The next result is classical, see~\cite[Satz 4']{szegoe_bemerkungen_grace} (or~\cite[Theorem~1.6]{marcus_spielman_srivastava}) for the first claim and~\cite[Satz 3]{szegoe_bemerkungen_grace} for the second one.
\begin{proposition}\label{prop:finite_free_preserces_positive_zeros_or_unit_circle}
If $P_n(x)$ and $Q_n(x)$ are two polynomials of degree $n$ which both have all their roots in $[0,\infty)$ (respectively, on the unit circle), then $P_n\boxtimes_n Q_n$ has all of its zeros in $[0,\infty)$ (respectively, on the unit circle).
\end{proposition}

\begin{proof}[Proof of Proposition~\ref{prop:mult_heat_acts_in_zeros}]
Write $P_n(x) = \sum_{j=0}^n a_{j;n} x^j$ and recall from~\eqref{eq:hermite_multipl_def}  that $H_n^*(x;s/n) = \sum_{j=0}^n (-1)^{n-j} \binom nj \eee^{-\frac s{2n} (j^2 - nj)} x^j$. Using~\eqref{eq:mult_finite_free_conv_def} we can write
\begin{equation}\label{eq:mult_heat_flow_as_finite_conv}
\eee^{- \frac s{2n} ((x \partial_x)^2- n x \partial_x)} P_n(x)
=
\sum_{j=0}^n a_{j;n} \eee^{- \frac s{2n} j^2 + \frac s2 j} x^j
=
P_n(x) \boxtimes_n H_n^*(x;s/n).
\end{equation}
It remains to apply Proposition~\ref{prop:mult_hermite_zeros_positive_unitary} in combination with Proposition~\ref{prop:finite_free_preserces_positive_zeros_or_unit_circle}.
\end{proof}

In the next theorem, we describe how the operator $\eee^{- \frac s{2n} ((x \partial_x)^2- n x \partial_x)}$ affects the asymptotic zero distribution of a degree $n$ polynomial, as $n\to\infty$. It is a multiplicative version of Theorem~\ref{theo:zeros_hermite_classical_asymptotic_distr}. Note the $1/n$-scaling in the time variable.

\begin{theorem}[Multiplicative heat flow and the asymptotic distribution of zeros]
Let $(P_n(z))_{n\in \N}$ be a sequence of polynomials  such that $\deg P_n = n$ for all $n\in\N$.
\begin{enumerate}
\item[(a)] The positive case: If each $P_n$ has only zeros in $[0,\infty)$ and $\llbracket  P_n \rrbracket_n$ converges  to some probability measure  $\rho$ weakly on $[0,\infty)$, then for every $s\geq 0$,  $\llbracket \eee^{- \frac s{2n} ((x \partial_x)^2- n x \partial_x)} P_n(x)\rrbracket_n$ converges  to $\rho \boxtimes \mu^{(s)}$ weakly on $[0,\infty)$. Here, $\boxtimes$ denotes the free multiplicative convolution of probability measures on $[0,\infty)$.
\item[(b)] The unitary case: If each $P_n$ has only zeros on the unit circle  $\{|z|=1\}$ and $\llbracket  P_n \rrbracket_n$ converges weakly to some probability measure  $\rho$ on $\{|z| = 1\}$, then for every $s\leq 0$,  $\llbracket  \eee^{- \frac s{2n} ((x \partial_x)^2- n x \partial_x) } P_n(x) \rrbracket_n$ converges weakly to $\rho \boxtimes \mu^{(s)}$. Here, $\boxtimes$ denotes the free multiplicative convolution of probability measures on the unit circle.
\end{enumerate}
\end{theorem}

\begin{proof}
Recall from~\eqref{eq:mult_heat_flow_as_finite_conv} that $\eee^{- \frac s{2n} ((x \partial_x)^2- n x \partial_x)} P_n(x) =P_n(x) \boxtimes_n H_n^*(x;s/n)$.

\smallskip
\noindent
\emph{Positive case:} Let $s\geq 0$. We know that $\llbracket  P_n \rrbracket_n \to \rho$ and $\llbracket  H_n^*(x; s/n)\rrbracket_n \to \mu^{(s)}$ weakly on $[0,\infty)$. It is known, see~\cite[Theorem~1.4]{arizmendi_garza_vargas_perales} or~\cite[Theorem~3.7]{jalowy_kabluchko_marynych_zeros_profiles_part_I}, that this implies  $\llbracket P_n(x) \boxtimes_n H_n^*(x;s/n)\rrbracket_n\to \rho \boxtimes \mu^{(s)}$ weakly on $[0,\infty)$.

\smallskip
\noindent
\emph{Unitary case:} Let $s\leq 0$. We know that $\llbracket  P_n \rrbracket_n \to \rho$ and $\llbracket  H_n^*(x; s/n)\rrbracket_n \to \mu^{(s)}$ weakly on $\{|z| = 1\}$.
By~\cite[Proposition~3.4]{arizmendi_garza_vargas_perales} (see also~\cite[Proposition~2.9]{kabluchko2024repeateddifferentiationfreeunitary} for the version needed here), this implies $\llbracket P_n(x) \boxtimes_n H_n^*(x;s/n)\rrbracket_n\to \rho \boxtimes \mu^{(s)}$ weakly on the unit circle.
\end{proof}

\section{Main results on multiplicative Laguerre polynomials}
\subsection{Definition of the multiplicative Laguerre polynomials}
The (associated) Laguerre polynomials with parameter $\alpha\in \R$ are defined by
$$
L_n^{(\alpha)}(x) = \sum_{j=0}^{n}\binom{n+\alpha}{n-j}\frac{(-x)^{j}}{j!}, \qquad n\in \N_0.
$$

Let $(\gamma_n)_{n\in\N}\subseteq \R$ be a sequence such that $\lim_{n\to\infty} \gamma_n /n = \gamma \geq -1$ and $\gamma_n \in [0,\infty) \cup\{-1,\ldots, -n+1\}$ for all $n\in \N$. It is well known, see, e.g., \cite[Section~4.7]{jalowy_kabluchko_marynych_zeros_profiles_part_II}, that the zero distribution $\llbracket L_n^{(\gamma_n)} (n \cdot  )\rrbracket_n$ converges to the Marchenko--Pastur law (the free Poisson distribution) with appropriate parameters.

Our aim is to prove a multiplicative version of this result.
\begin{definition}
For parameters $c\in \N_0$ and $b\in \C$, the \emph{multiplicative Laguerre polynomials} $L_n^*(x; b,c)$ are defined by
\begin{equation}\label{eq:laguerre_mult_def}
L_n^*(x; b,c) := (x\partial_x + b)^c (x-1)^n =  \sum_{j=0}^n (-1)^{n-j} \binom nj (j+b)^c x^j, \qquad n\in \N_0.
\end{equation}
\end{definition}
To motivate this definition, let us mention that these polynomials appear as expected characteristic polynomials of random matrix \emph{products} $A_{1;n}\cdot \ldots \cdot A_{c;n}$, where each $A_{j;n} = \mathrm{Id}_n  - \frac{n}{b+n} U_{j;n} U_{j;n}^\top$ is a rank-one perturbation of the $n\times n$-identity matrix $\mathrm{Id}_n$, and $U_{1;n},\ldots, U_{c;n}$ are independent random vectors uniformly distributed on the unit sphere in $\R^n$; see~\cite[Section~2.4]{kabluchko2024repeateddifferentiationfreeunitary} and~\cite[Remark~5.13]{jalowy_kabluchko_marynych_zeros_profiles_part_II} for exact statements. Similarly, the associated Laguerre polynomials are expected characteristic polynomials of $\sum_{j=1}^c U_{j;n} U_{j;n}^\top $, i.e.\ \emph{sums} of random rank-one matrices; see~\cite[Section~6.3]{marcus2021polynomial}.
Note also that~\eqref{eq:laguerre_mult_def} is the multiplicative analogue (with $\partial_x$ replaced by $x\partial_x$) of the following property of the associated Laguerre polynomials:
\[
\begin{aligned}
&\text{if } n\ge m:\quad (\partial_x+b)^m x^n  =  m! x^{ n-m} L_m^{(n-m)}(-b x),\\
&\text{if } m\ge n:\quad (\partial_x+b)^m x^n  =  n! b^{ m-n} L_n^{(m-n)}(-b x).
\end{aligned}
\]
The multiplicative Laguerre polynomials are closely related to the Sidi polynomials appearing in~\cite[Theorems~4.2, 4.3]{sidi_numerical_quad_nonlinear_seq}, \cite[Theorem~3.1]{sidi_num_quad_infinite_range}; see also see~\cite[Theorem~1.3]{lubinski_sidi_strong_asympt}, \cite[Theorem~2]{lubinski_stahl_some_explicit} and~\cite{lubinski_sidi_composite_biorthogonal} for their asymptotic zero distribution.

\subsection{Zeros of the multiplicative Laguerre polynomials}
The following result outlines the conditions on the parameters $b$ and $c$ that ensure the polynomial $L_n^*(\cdot; b,c)$  has only positive/unitary zeros.
\begin{proposition}[Positive and unitary zeros]\label{prop:mult_laguerre_positive_unitary_zeros}
Let $n\in \N$.
\begin{itemize}
\item[(a)]
The positive case: For $c\in \N_0$ and $b\in \R \backslash [-n,0]$,  all zeros of $L_n^*(\cdot; b,c)$ are real and positive.
\item[(b)]
The unitary case: For $c\in \N_0$ and $b\in \C$ with $\Re b = -n/2$, all zeros of $L_n^*(\cdot; b,c)$ are located on the unit circle $\{|z| =1\}$.
\end{itemize}
\end{proposition}
\begin{proof}
For $c=0$ we have $L_n^*(x; b,0)= (x-1)^n$ and both claims are true. Let now $c=1$. Using the binomial theorem and the identity $j\binom{n}{j}=n\binom{n-1}{j-1}$, we have
\[
\begin{aligned}
L_n^*(x; b,1)
&
=b\sum_{j=0}^n (-1)^{ n-j}\binom{n}{j}x^j
  +\sum_{j=0}^n (-1)^{ n-j}\binom{n}{j}j x^j
=b(x-1)^n
  +\sum_{j=1}^n (-1)^{ n-j} n\binom{n-1}{j-1}x^j
\\
&=b(x-1)^n + n x (x-1)^{n-1}
=(x-1)^{n-1}\left((b+n)x-b\right).
\end{aligned}
\]
Thus, for $b\in \C \backslash\{-n\}$, the zeros of $L_n^*(x; b,1)$ are $x=1$ (with multiplicity $n-1$) and $x=\frac{b}{b+n}$ (with multiplicity $1$).
To complete the proof for $c=1$, observe (a): if $b\in \R \backslash [-n,0]$, then $\frac{b}{b+n}>0$, and (b): if $b \in  -\frac n2 + \ii   \R$, then $\frac{b}{b+n}$ belongs to the unit circle.

To prove the proposition for general $c\in \N$, write $L_n^*(x; b,c) =  L_n^*(x; b,1) \boxtimes_n \ldots \boxtimes_n L_n^*(x; b,1)$ (with $c$ factors on the right-hand side) and apply Proposition~\ref{prop:finite_free_preserces_positive_zeros_or_unit_circle}.
\end{proof}

\begin{theorem}[Asymptotic zero distribution of multiplicative Laguerre polynomials]\label{theo:multiplicative_laguerre_polys_zeros}
Let $(b_n)_{n\in \N}\subseteq \C$ and $(c_n)_{n\in \N}\subseteq \N$ be sequences such that $b_n/n\to \beta$ and $c_n/n \to \gamma$ as $n\to\infty$, where $\beta\in \C$ and $\gamma>0$ are constants.
\begin{itemize}
\item[(a)] The positive case: If $b_n\in \R\backslash[-n,0]$ for every $n\in \N$ and  $\beta \in \R \backslash [-1,0]$, then  the probability measure $\llbracket  L_n^*(\cdot; b_n, c_n)\rrbracket_n$ converges weakly to some probability measure $\nu_{\beta, \gamma}$ on $[0,\infty)$.
\item[(b)] The unitary case: If $\Re b_n = -n/2$ for every $n\in \N$ (and hence $\Re \beta = -1/2$), then the probability measures $\llbracket  L_n^*(\cdot; b_n, c_n)\rrbracket_n$ converge weakly to some probability measure $\nu_{\beta, \gamma}$ on the unit circle. %$\{|z| = 1\}$.
\item[(c)]
In both cases, the $S$-transform of $\nu_{\beta, \gamma}$ is given by
\begin{equation}\label{eq:S_transform_mult_poisson}
S(y)=\exp\left(\frac{\gamma}{\beta+1+y}\right).
\end{equation}
\end{itemize}
\end{theorem}

A proof in the positive case has been given in~\cite[Theorem~5.12]{jalowy_kabluchko_marynych_zeros_profiles_part_II} using exponential profiles.  For $b_n = -n/2$, the statement has been proven in~\cite[Theorem~2.7]{kabluchko2024repeateddifferentiationfreeunitary} using the saddle-point method. In Section~\ref{sec:proof_laguerre}, we shall prove Theorem~\ref{theo:multiplicative_laguerre_polys_zeros} in full generality using a new  method which allows to treat the positive and the unitary cases in a unified way. After minimal modifications,  the same method proves convergence of analytic moments of $\llbracket  L_n^*(\cdot; b_n, c_n)\rrbracket_n$ assuming only $b_n/n \to \beta \in \C\backslash\{-1\}$  and $c_n/c\to \gamma >0$. If the roots of $L_n^*(x; b_n, c_n)$ are positive or unitary, this implies weak convergence. In general, we have the following

\begin{conjecture}
Let $(b_n)_{n\in \N}\subseteq \C$ and $(c_n)_{n\in \N}\subseteq \N$ be sequences such that $b_n/n\to \beta$ and $c_n/n \to \gamma$ as $n\to\infty$, where $\beta\in \C$ and $\gamma>0$ are constants (without further restrictions on $\beta$).  Then, $\llbracket  L_n^*(\cdot; b_n, c_n)\rrbracket_n$ converges (as $n\to\infty$) to some probability measure $\nu_{\beta, \gamma}$, weakly on $\C$.  The analytic moments of $\nu_{\beta, \gamma}$ are given by~\eqref{eq:nu_moments_1} and~\eqref{eq:nu_moments_2}, below.
\end{conjecture}

The measures $\nu_{\beta,\gamma}$ appearing in Theorem~\ref{theo:multiplicative_laguerre_polys_zeros} are the free multiplicative analogues of the Poisson distribution that appeared in the work of Bercovici and Voiculescu~\cite{bercovici_voiculescu_levy_hincin}; see~\cite[Lemma~6.4]{bercovici_voiculescu_levy_hincin}  for the unitary case when $\Re \beta = -1/2$ and  \cite[Lemma~7.2]{bercovici_voiculescu_levy_hincin} for the positive case when $\beta \in \R \backslash [-1,0]$.
In the next result we collect  some properties of $\nu_{\beta, \gamma}$.
\begin{proposition}[Free multiplicative Poisson distributions]\label{prop:mult_free_poisson_properties}
Let $\gamma>0$ and $\beta\in \C$ be such that either $\beta \in \R \backslash [-1,0]$ or $\Re \beta = -1/2$. Then, there exists a unique probability measure  $\nu_{\beta, \gamma}$ which is supported on a compact subset of $(0,\infty)$ (when $\beta \in \R \backslash [-1,0]$) or on the unit  circle (when $\Re \beta = -1/2$) and has the $S$-transform $S(y)=\exp (\frac{\gamma}{\beta+1+y})$. Moreover, $\nu_{\beta, \gamma}$ has the following properties.
\begin{itemize}
\item[(a)] Moments: For every $k\in \N$, the $k$-th moment of $\nu_{\beta, \gamma}$ is given by
\begin{align}
\int_\C x^k \nu_{\beta, \gamma}(\dd x)
&=
\frac1k [t^{ k-1}] \left( (1+t)^k \eee^{-k \gamma/(1+\beta+t)}\right) \label{eq:nu_moments_1}
\\
&=
\frac {\eee^{- k\gamma/(1 + \beta)}}k  \sum_{j=0}^{k-1}\binom{k}{j+1}\left(\frac {-1}{1+\beta}\right)^{j}
L_{j}^{(-1)}\left(\frac {k\gamma}{1+\beta}\right), \label{eq:nu_moments_2}
\end{align}
where $L_n^{(-1)}(c)= \sum_{j=1}^{n} \binom{n-1}{n-j}\frac{(-c)^{j}}{j!}$, for $n\in \N$, and $L_0^{(-1)}(c) = 1$,  are the associated Laguerre polynomials with $\alpha = -1$.
\item[(b)] Free cumulants: For every $k\in \N$, the $k$-th free cumulant of $\nu_{\beta, \gamma}$ is given by
\begin{align}
\kappa_k
=
\frac{1}{k} [t^{ k-1}] \eee^{-k \gamma/(1+\beta+t)}
=
\frac {\eee^{- k\gamma/(1 + \beta)}}k \left(\frac {-1}{1+\beta}\right)^{k-1}
L_{k-1}^{(-1)}\left(\frac {k\gamma}{1+\beta}\right).  \label{eq:nu_cumulants}
\end{align}
\end{itemize}
\end{proposition}
Existence of $\nu_{\beta, \gamma}$ has been established in~\cite{bercovici_voiculescu_levy_hincin}; the proof of the remaining properties  will be given in Section~\ref{subsec:free_mult_poisson_properties}. In Lemma~\ref{lem:m_k_gamma_ODE_laguerre}, we shall state a system of ODE's satisfied by the moments of $\nu_{\beta, \gamma}$.

\begin{remark}\label{rem:mult_laguerre_symmetry}
It is easy to check that $L_n^*(x; -n-b, c)=(-1)^{ n+c} x^{ n} L_n^*(x^{-1}; b, c)$. It follows from Theorem~\ref{theo:multiplicative_laguerre_polys_zeros} that $\nu_{-1-\beta, \gamma}$ is the image of $\nu_{\beta, \gamma}$ under the map $x\mapsto 1/x$, for all admissible parameters $\beta, \gamma$.
\end{remark}

%\begin{remark}[On $\beta = 0$ and $\beta = -1$] In Theorem~\ref{theo:multiplicative_laguerre_polys_zeros}, we excluded these two cases (which can be reduced to each other by Remark~\ref{rem:mult_laguerre_symmetry}). The reason for excluding $\beta = -1$ is that in the proof, the expression $1/(\beta+1)$ appears many times. Additionally, for $\beta = 0$ (more precisely, for $b_n = 0$), the degree of $L_n^*(x; 0, c_n)$ is less than $n$. Thus, our method does not allow to treat the polynomials $L_n^*(x; 0, c_n) =  \sum_{j=0}^n (-1)^{n-j} \binom nj j^{c_n} x^j$ and $L_n^*(x; -n, c_n) = \sum_{j=0}^n (-1)^{n-j} \binom nj (j-n)^{c_n} x^j$, where $c_n/n\to \gamma >0$. Still,  the asymptotic zero distribution of $L_n^*(x; 0, c_n)$ has been determined  in~\cite[Theorem~4.2]{jalowy_kabluchko_marynych_zeros_profiles_part_I} (taking $a=b=1$ there), see also~\cite[Theorem~5.12]{jalowy_kabluchko_marynych_zeros_profiles_part_II}. Using Remark~\ref{rem:mult_laguerre_symmetry}, it is easy to  deduce the asymptotic zero distribution of  $L_n^*(x; -n, c_n)$. It has been pointed to us by Daniel Perales that the operator $x\partial_x - n$ (which corresponds to the choice $b_n = -n$) is up to sign the $0$-polar derivative studied in~\cite[Theorem~1.1]{perales2025asymptoticrootdistributionpolynomials}.
%\end{remark}

\begin{remark}[On $\beta=0$ and $\beta=-1$]
%In Theorem~\ref{theo:multiplicative_laguerre_polys_zeros} we excluded the cases $\beta=-1$ and $\beta=0$, which are equivalent by Remark~\ref{rem:mult_laguerre_symmetry}. The exclusion of $\beta=0$ is purely technical. (It makes zeros strictly positive -- the core argument expends to the case $\beta = 0$) By contrast, for $\beta=-1$ the proof breaks down completely, since the factor $1/(\beta+1)$ occurs throughout the argument. Consequently, our approach does not apply to the families
In Theorem~\ref{theo:multiplicative_laguerre_polys_zeros} we excluded the cases $\beta=-1$ and $\beta=0$, which are equivalent by Remark~\ref{rem:mult_laguerre_symmetry}. The exclusion of $\beta=0$ is purely technical: it ensures that all zeros are strictly positive, and the argument extends to $\beta=0$ with only minor modifications. By contrast, for $\beta=-1$ the proof breaks down completely, since the factor $1/(\beta+1)$ appears throughout the argument. Consequently, our approach does not apply to the polynomials
\[
L_n^*(x;0,c_n)=\sum_{j=0}^n (-1)^{n-j}\binom{n}{j}\, j^{c_n}x^j,
\qquad
L_n^*(x;-n,c_n)=\sum_{j=0}^n (-1)^{n-j}\binom{n}{j}\,(j-n)^{c_n}x^j,
\]
in the regime $c_n/n\to\gamma>0$. The asymptotic zero distribution of $L_n^*(x;0,c_n)$ was determined in~\cite[Theorem~4.2]{jalowy_kabluchko_marynych_zeros_profiles_part_I} (with $a=b=1$ there); see also~\cite[Theorem~5.12]{jalowy_kabluchko_marynych_zeros_profiles_part_II}. The limiting measure has an atom of mass $\max(1-\gamma,0)$ at $x=1$ and an absolutely continuous part of mass $\min(\gamma,1)$ with Lebesgue density
\[
x\mapsto -\frac{1}{\pi x}\,\Im\!\left(\frac{1}{1+\gamma^{-1}W_0\!\left(-\gamma \eee^{-\gamma}x^{-1}+\ii 0\right)}\right),
\qquad x\in\bigl(0,\gamma \eee^{1-\gamma}\bigr),
\]
where $W_0$ denotes the principal branch of the Lambert $W$-function. By Remark~\ref{rem:mult_laguerre_symmetry}, the asymptotic zero distribution of $L_n^*(x;-n,c_n)$ is the push-forward of this measure under the map $x\mapsto 1/x$. In particular, the resulting measure has unbounded support, which provides another reason why the method of moments used in the proof of Theorem~\ref{theo:multiplicative_laguerre_polys_zeros} breaks down.

\end{remark}

\subsection{Repeated action of $x \partial_x + b$}
If $P(x)$ is a real-rooted polynomial, then its derivatives $P'(x), P''(x), \ldots$ are also real-rooted -- this follows from Rolle's theorem. The next result is a similar claim for the repeated action of the operator $x\partial_x + b$.

\begin{proposition}[$(x \partial_x + b)^c$ preserves positive/unitary zeros]\label{prop:mult_derivative_acts_in_zeros}
Let $P_n(x)\in \C[x]$ be a polynomial of degree $n\in \N$. Let $c\in \N_0$ and $b\in \C$.
\begin{enumerate}
\item[(a)]  The positive case: If  all roots of $P_n$ are in $[0,\infty)$ and $b\in \R\backslash[-n,0]$,  then all roots of the polynomial $(x \partial_x + b)^c P_n$ are also in $[0,\infty)$.
\item[(b)] The unitary case: If  all roots of $P_n$ are located  on the unit circle $\{|z|=1\}$ and $\Re b = -n/2$,  then  all roots of the polynomial $(x \partial_x + b)^c P_n$ are also located on the unit circle.
\end{enumerate}
\end{proposition}
\begin{proof}
Let $P_n(x) = \sum_{j=0}^n a_{j;n} x^j$, then using~\eqref{eq:mult_finite_free_conv_def} we can write
\begin{equation}\label{eq:mult_differentiation_flow_as_finite_conv}
(x \partial_x + b)^c P_n(x)
=
\sum_{j=0}^n a_{j;n} (j + b)^c x^j
=
P_n(x) \boxtimes_n L_n^*(x; b,c).
\end{equation}
It remains to apply Proposition~\ref{prop:mult_laguerre_positive_unitary_zeros} in combination with Proposition~\ref{prop:finite_free_preserces_positive_zeros_or_unit_circle}.
\end{proof}

Steinerberger~\cite{Steiner21} discovered that if $(P_n(x))_{n\in \N}$ is a sequence of real-rooted polynomials with $\deg P_n = n$ and such that $\llbracket P_n\rrbracket_n$ converges to some probability measure $\rho$, then for every fixed $t\in (0,1)$, the $[tn]$-th derivative of $P_n$ also has some limiting distribution of zeros which can be described in terms of free probability. Proofs have been given in~\cite[Theorems~4.1, 4.4]{hoskins_kabluchko} and~\cite[Theorem~3.7]{arizmendi_garza_vargas_perales}.  A more general  result on  the repeated action of the operator $x^a \partial_x^b$ can be found in~\cite[Section~4.1]{jalowy_kabluchko_marynych_zeros_profiles_part_I}.  The next theorem is an analogous result for the repeated action of $x\partial_x + b$.

\begin{theorem}[Repeated action of $x\partial_x + b$ and the asymptotic distribution of zeros]
Let $(P_n(x))_{n\in \N}$ be a sequence of polynomials  such that $\deg P_n = n$ for all $n\in\N$. Let also $(b_n)_{n\in \N}\subseteq \C$ and $(c_n)_{n\in \N}\subseteq \N$ be sequences such that $b_n/n\to \beta$ and $c_n/n \to \gamma$ as $n\to\infty$, where $\beta\in \C$ and $\gamma>0$ are constants.
\begin{enumerate}
\item[(a)] The positive case: If each $P_n$ has only zeros in $[0,\infty)$ and $\llbracket  P_n \rrbracket_n$ converges  to some probability measure $\rho$ weakly on $[0,\infty)$, and if $b_n\in \R\backslash[-n,0]$ for every $n\in \N$ and $\beta \in \R \backslash [-1,0]$,  then   $\llbracket (x \partial_x + b_n)^{c_n} P_n(x)\rrbracket_n$ converges  to $\rho \boxtimes \nu_{\beta, \gamma}$ weakly on $[0,\infty)$. Here, $\boxtimes$ denotes the free multiplicative convolution of probability measures on $[0,\infty)$.
\item[(b)] The unitary case: If each $P_n$ has only zeros on the unit circle  $\{|z|=1\}$ and $\llbracket  P_n \rrbracket_n$ converges weakly to some probability measure  $\rho$ on $\{|z| = 1\}$,  and if $\Re b_n = -n/2$ for every $n\in \N$ (implying $\Re \beta = -1/2$), then   $\llbracket (x \partial_x + b_n)^{c_n}  P_n(x) \rrbracket_n$ converges weakly to $\rho \boxtimes \nu_{\beta, \gamma}$. Here, $\boxtimes$ denotes the free multiplicative convolution of probability measures on the unit circle.
\end{enumerate}
\end{theorem}
\begin{proof}
Recall from~\eqref{eq:mult_differentiation_flow_as_finite_conv} that $(x \partial_x + b_n)^{c_n} P_n(x) =P_n(x) \boxtimes_n L_n^*(x; b_n,c_n)$.

\smallskip
\noindent
\emph{Positive case}: We know that $\llbracket  P_n \rrbracket_n \to \rho$ and $\llbracket  L_n^*(x; b_n,c_n)\rrbracket_n \to \nu_{\beta, \gamma}$ weakly on $[0,\infty)$. By known results, see~\cite[Theorem~1.4]{arizmendi_garza_vargas_perales}  or~\cite[Theorem~3.7]{jalowy_kabluchko_marynych_zeros_profiles_part_I}, this implies  $\llbracket P_n(x) \boxtimes_n L_n^*(x; b_n,c_n)\rrbracket_n\to \rho \boxtimes \nu_{\beta, \gamma}$ weakly on $[0,\infty)$.

\smallskip
\noindent
\emph{Unitary case}:  We know that $\llbracket  P_n \rrbracket_n \to \rho$ and $\llbracket  L_n^*(x; b_n,c_n)\rrbracket_n \to \nu_{\beta, \gamma}$ weakly on $\{|z| = 1\}$.
By~\cite[Proposition~3.4]{arizmendi_garza_vargas_perales} (see also~\cite[Proposition~2.9]{kabluchko2024repeateddifferentiationfreeunitary} for the version needed here), this implies $\llbracket P_n(x) \boxtimes_n L_n^*(x; b_n,c_n)\rrbracket_n\to \rho \boxtimes \nu_{\beta, \gamma}$ weakly on the unit circle.
\end{proof}

\begin{remark}
As Daniel Perales pointed out to us, the operator $x\partial_x-n$ (corresponding to $b_n=-n$) agrees, up to an overall sign, with the $0$-polar derivative studied in~\cite[Theorem~1.1]{perales2025asymptoticrootdistributionpolynomials}.
\end{remark}

\section{Proof of Theorem~\ref{theo:multiplicative_hermite_polys_zeros}}\label{sec:proof_hermite}
The present section is devoted to the proof of Theorem~\ref{theo:multiplicative_hermite_polys_zeros}.
\subsection{Setup}\label{subsec:setup_hermite}
Let us fix the notation.
For $s\in \C$ (in most cases assuming $s\in \R$) we consider the polynomials
$$
P_n(x;s)   = \eee^{- \frac s{2n} (x \partial_x)^2 } (x-1)^n = \sum_{j=0}^n (-1)^{n-j} \binom{n}{j} \eee^{- s \frac{j^2}{2n} } x^j = H_n^*(\eee^{-s/2} x; s/n).
$$
Let  $x_{1;n}(s)\in \C,\dots,x_{n;n}(s)\in \C$ be the zeros of the polynomial $\Pn(\cdot;s)$ (counted with multiplicity), so that
\begin{equation}\label{eq:root-factorization-n}
\Pn(x;s) =   \eee^{-ns/2} \prod_{j=1}^n\left(x-x_{j;n}(s)\right),
\;
\llbracket  P_n(\cdot; s)\rrbracket_n = \frac 1n \sum_{j=1}^n  \delta_{x_{j;n}(s)},
\;
\llbracket H_n^*(\cdot; s/n)\rrbracket_n = \frac 1n \sum_{j=1}^n  \delta_{\eee^{-s/2} x_{j;n}(s)}.
\end{equation}
We denote the $k$-th moment of the zero distribution  $\llbracket  P_n(\cdot; s)\rrbracket_n$ by
\begin{equation}\label{eq:sigma-def-n}
\sg_{k;n}(s) := \frac1n\sum_{j=1}^n x_{j;n}(s)^k,
\quad k\in \N_0.
\end{equation}

Note that $\sg_{0;n}(s)= 1$ for all $s\in \R$, and $\sigma_{k; n}(0) = 1$ for all $k\in  \mathbb N$.
Observe that for every fixed $n\in \mathbb N$ and $k\in \mathbb N_0$, the function $\sigma_{k;n}(s)$ is analytic in $s\in \mathbb C$. Indeed, using Newton's identities, we can express $\sigma_{k;n}(s)$ as a polynomial in the coefficients of $\eee^{ns/2} P_n(x;s)$, which  are analytic functions of  $s$.

Define the generating function
\begin{equation}\label{eq:Sn-def}
\Sigma_n(x;s) := \sum_{k=0}^\infty\sg_{k;n}(s) x^{-k}
 = \frac1n\sum_{j=1}^n\frac{1}{1-x_{j;n}(s)x^{-1}}
 = \frac xn\sum_{j=1}^n\frac{1}{x-x_{j;n}(s)}.
\end{equation}
Recognizing on the right-hand side the logarithmic derivative of $P_n$, we can write
\begin{equation}\label{eq:Sn-from-Pn-for-Hermite}
\Sigma_n(x;s) = \frac{x}{n} \frac{\partial_x \Pn(x ;s)}{\Pn(x ;s)}.
\end{equation}
It follows from~\eqref{eq:Sn-def} that  $\Sigma_n(\infty; s) := \lim_{|x|\to\infty} \Sigma_n(x;s) = 1$.

\begin{remark}\label{rem:analyticity_sigma_cauchy_bound}
For every $n\in \mathbb N$ and $A>0$, there is a number $C_{n,A}$ such that the function  $x \mapsto \Sigma_n(x;s)$ is analytic on $\{x\in \mathbb C: |x|> C_{n,A}\}$ provided $|s|\leq A$. Indeed, by Cauchy's bound, all roots of $P_n(x;s)$ have absolute value smaller than
$$
C_{n,A} := 2 + \max_{j=0,\ldots, n-1} \binom nj|\eee^{-\frac{s}{2n} (j^2-n^2)}| \leq  2+ 2^n \eee^{A n/2}.
$$

\end{remark}

\subsection{Plan of the proof}
To prove that the zero distributions $\llbracket  P_n(\cdot; s)\rrbracket_n$ converge as $n\to\infty$, we shall argue by the method of moments. Our aim is to show that for every $k\in \N_0$ and $s\in \R$,
\begin{equation}\label{eq:hermite_proof_moments_converge_aim}
\lim_{n\to\infty} \sigma_{k; n} (s) =  \frac{\eee^{sk}}{k} [u^{ k-1}] \left((1+u)^{k} \eee^{ sk u}\right).
\end{equation}
By~\eqref{eq:root-factorization-n},  the $k$-th moment of $\llbracket H_n^*(\cdot; s/n)\rrbracket_n$ is $\eee^{-sk/2} \sg_{k;n}(s)$, which gives~\eqref{eq:mu_s_moments1}.

The proof of~\eqref{eq:hermite_proof_moments_converge_aim} (and Theorem~\ref{theo:multiplicative_hermite_polys_zeros}) proceeds in the following steps.
\begin{itemize}
\item[(1)] We derive a multiplicative heat PDE satisfied by $P_n(x;s)$.
\item[(2)] We use it to derive a PDE for $\Sigma_n(x;s) = \sum_{k=0}^\infty\sg_{k;n}(s) x^{-k} =  \frac{x}{n} \frac{\partial_x \Pn(x ;s)}{\Pn(x ;s)}$.
\item[(3)] Using coefficient extraction, we derive a triangular system of ODE's satisfied by the functions $\sigma_{k;n}(s)$, $k\in \N_0$.
\item[(4)] We show that the functions $\sigma_{k;\infty}(s) := \lim_{n\to\infty} \sigma_{k;n}(s)$, $k\in \N_0$,  exist and satisfy a triangular system of ODE's which is a natural continuum limit (as $n\to\infty$) of the ODE system satisfied  by the functions $\sigma_{k;n}(s)$, $k\in \N_0$.
\item[(5)] We show that the right-hand side of~\eqref{eq:hermite_proof_moments_converge_aim} satisfies the same system of ODE's as the $\sigma_{k;\infty}(s)$'s. Uniqueness of solution proves~\eqref{eq:hermite_proof_moments_converge_aim}.
\item[(6)] We conclude the proof of Theorem~\ref{theo:multiplicative_hermite_polys_zeros} by showing that the convergence of moments implies the weak convergence of the zero distributions.
\end{itemize}

\subsection{PDE for the polynomial.}
In the following, $\partial_x$ and $\partial_s$ denote partial derivative operators in the variables $x$ and $s$. Let  $L=x\partial_x$.
\begin{proposition}
Let $n\in \mathbb N$. The function $P_n(x;s)$
obeys the PDE
$$
\begin{cases}
\partial_s P_n (x;s) = - \frac 1{2n}  L^2 P_n(x;s), \qquad s\in \mathbb R, \quad  x\in \mathbb C,
\\
P_n(x;0) = (x-1)^n.
\end{cases}
$$
\end{proposition}
\begin{proof}
Clearly, $L x^k = (x \partial_x) x^k = k x^k$ and $L^2 x^k = (x \partial_x)^2 x^k = k^2 x^k$. Therefore,
$$
\partial_s ( \eee^{- s \frac{k^2}{2n}}   x^k )
=
- \frac 1 {2n} \eee^{- s \frac{k^2}{2n}}  k^2 x^k
=
- \frac 1 {2n} \eee^{- s \frac{k^2}{2n}}  L x^k
=
- \frac 1 {2n} L^2 \left(\eee^{- s \frac{k^2}{2n}} x^k\right).
$$
By linearity of $\partial_s$ and $L$, this gives $\partial_s P_n (x;s) = - \frac 1{2n}  L^2 P_n(x;s)$.
\end{proof}

\subsection{PDE for the generating function of the finite-$n$ moments.}
Let $n\in \mathbb N$. We claim that $\Sigma_n(x;s)$ satisfies the following (viscous) Burgers-type PDE:
$$
\partial_s \Sigma_n  =  -\tfrac12 L(\Sigma_n^2) - \tfrac{1}{2n} L^2\Sigma_n,\qquad L=x\partial_x.
$$
The range of the variables is: $s\in (-A,A)$, $|x|>C_{n,A}$, with arbitrary $A>0$; see Remark~\ref{rem:analyticity_sigma_cauchy_bound}.
\begin{proof}
We use  the PDE for $P_n(x;s)$ together with~\eqref{eq:Sn-from-Pn-for-Hermite}.
Put $u_n(x;s):=\log P_n(x;s)$. Then $\partial_s P_n= -\frac{1}{2n}L^2P_n$ implies
\begin{equation}\label{eq:ODE_u_n}
\partial_s u_n = \frac{\partial_s P_n}{P_n} = -\frac 1 {2n} \frac{L^2 P_n}{P_n} = -\frac 1 {2n} \frac{L^2 (\eee^{u_n})}{\eee^{u_n}}
=-\frac{1}{2n}\left((Lu_n)^2+L^2u_n\right),
\end{equation}
since $L\eee^{u_n} = x\partial_x (\eee^{u_n}) =   x \eee^{u_n} (\partial_x u_n) = \eee^{u_n} L u_n$ and hence
$$
L^2(\eee^{u_n}) = L (\eee^{u_n} L u_n) = L(\eee^{u_n}) L u_n +  \eee^{u_n} L^2 u_n
=\eee^{u_n}\left((Lu_n)^2+L^2u_n\right).
$$
By definition, $\Sigma_n=\frac{x}{n}\frac{\partial_x P_n}{P_n}=\frac{1}{n}Lu_n$. Hence $Lu_n = n\Sigma_n$ and $L^2u_n =n L\Sigma_n$. It follows that
$$
\partial_s \Sigma_n = \frac 1n \partial_s Lu_n =  \frac 1n L \partial_s u_n \stackrel{\eqref{eq:ODE_u_n}}{=} -\frac 1 {2n^2} L \left((Lu_n)^2+L^2u_n\right) = -\frac 1{2n^2} L (n^2 \Sigma^2_n + n L \Sigma_n),
$$
which gives the claimed PDE after cancellation.
\end{proof}

\begin{remark}[Log-coordinate form]
With $y=\log x$ and $\tilde \Sigma_n(y,s):=\Sigma_n(\eee^y;s)$, the equation becomes
$$
\partial_s \tilde \Sigma_n(y;s) + \tilde \Sigma_n(y;s)  \partial_y \tilde \Sigma_n(y;s)
= -\frac{1}{2n} \partial_y^2\tilde \Sigma_n(y;s),
$$
i.e.\ the classical Burgers equation with (here, negative) viscosity $-1/(2n)$.
\end{remark}

\subsection{Finite-$n$ ODE system for power sums.}\label{subsec:finite_n_ODE}
Let $n\in \mathbb N$ be fixed. We claim that the functions $\sigma_{k;n}(\cdot)$, $k\in \mathbb N_0$, solve the ODE system
\begin{equation}\label{eq:ODE_system_finite_n}
\begin{aligned}
&\sigma_{k;n}'(s)
=\frac{k}{2}\sum_{j=0}^{k}\sigma_{j;n}(s)\sigma_{k-j;n}(s) - \frac{k^2}{2n} \sigma_{k;n}(s), \qquad k\in \mathbb N, \qquad s\in \mathbb R,
\\
&\sigma_{0;n}(s) =  1 \quad (\text{for all } s\in \mathbb R), \qquad \sigma_{k; n}(0) = 1 \quad (\text{for all } k\in  \mathbb N).
\end{aligned}
\end{equation}
\begin{example}
The first three equations are
$$
\begin{aligned}
\sigma_{1;n}'=(1-\tfrac{1}{2n}) \sigma_{1;n},
\quad
\sigma_{2;n}'=(2-\tfrac{4}{2n})\sigma_{2;n}+\sigma_{1;n}^2,
\quad
\sigma_{3;n}'=(3-\tfrac{9}{2n})\sigma_{3;n}+\tfrac{3}{2}\sigma_{1;n}\sigma_{2;n}.
\end{aligned}
$$
\end{example}
\begin{remark}
The solution to~\eqref{eq:ODE_system_finite_n}  is unique.  In fact, the system can be solved successively starting with $\sigma_{0;n}\equiv 1$: If we already know $\sigma_{1;n}(s), \ldots, \sigma_{k-1;n}(s)$, then the equation for $\sigma_{k;n}'$ determines $\sigma_{k;n}$ uniquely.
\end{remark}
\begin{proof}[Proof of~\eqref{eq:ODE_system_finite_n}]
Start from the  PDE
\begin{equation}\label{eq:PDE_for_Sigma_proof_tech1}
\partial_s \Sigma_n  =  -\tfrac12 L(\Sigma^2_n) - \tfrac{1}{2n} L^2\Sigma_n,\qquad L=x\partial_x.
\end{equation}
Expand
$$
\Sigma_n(x;s)=\sum_{k = 0}^\infty\sigma_{k;n}(s) x^{-k},\qquad
\Sigma^2_n(x;s)=\sum_{k = 0}^\infty c_{k;n}(s) x^{-k},\quad
c_{k;n}(s):=\sum_{j=0}^{k}\sigma_{j;n}(s)\sigma_{k-j;n}(s).
$$
Since $L x^{-k}=-k x^{-k}$ and $L^2 x^{-k}=k^2 x^{-k}$, we get
$$
L(\Sigma^2_n) = \sum_{k = 0}^\infty(-k) c_{k;n}(s) x^{-k},\qquad
L^2\Sigma_n = \sum_{k = 0}^\infty k^2 \sigma_{k;n}(s) x^{-k}.
$$
Coefficient extraction of $x^{-k}$ in the PDE~\eqref{eq:PDE_for_Sigma_proof_tech1} yields
$$
\sigma_{k;n}'(s)=\frac{k}{2} c_{k;n}(s) - \frac{k^2}{2n} \sigma_{k;n}(s)
=\frac{k}{2}\sum_{j=0}^{k}\sigma_{j;n}(s)\sigma_{k-j;n}(s) - \frac{k^2}{2n} \sigma_{k;n}(s),
$$
for every $k\in \mathbb N$. The initial conditions follow from the definition of $\sigma_{k;n}$.
\end{proof}

\subsection{Formal inviscid limit.}
Consider an ODE system obtained from~\eqref{eq:ODE_system_finite_n} by taking the \emph{formal} $n\to\infty$ limit:
\begin{equation}\label{eq:ODE_system_infinite_n}
\begin{aligned}
&\sigma_{k;\infty}'(s)
=\frac{k}{2}\sum_{j=0}^{k}\sigma_{j;\infty}(s)\sigma_{k-j;\infty}(s), \qquad k\in \mathbb N, \qquad s\in \mathbb R,
\\
&\sigma_{0;\infty}(s) =  1 \quad (\text{for all } s\in \mathbb R), \qquad \sigma_{k; \infty}(0) = 1 \quad (\text{for all } k\in  \mathbb N).
\end{aligned}
\end{equation}
This system, which can be integrated successively,  admits a unique solution and  defines a sequence of functions $\sigma_{k;\infty}(s)$, $k\in \mathbb N_0$.

\subsection{Forced ODE systems.}
Note that in the equation for $\sigma_{k;n}'$, the right-hand side contains the term  $(k-\frac{k^2}{2n})\sigma_{k;n}$.
To remove it, we introduce the functions $\eta_{k;n}(s)$ and $\eta_{k;\infty}(s)$ by
\begin{equation}\label{eq:etas_def}
\sigma_{k;n}(s)=\eee^{\left(k-\frac{k^2}{2n}\right)s} \eta_{k;n}(s),
\qquad
\sigma_{k;\infty}(s)=\eee^{k s} \eta_{k;\infty}(s),
\qquad k\in \mathbb N_0, \qquad s\in \mathbb R.
\end{equation}
In the finite-$n$ case, the ODE system~\eqref{eq:ODE_system_finite_n} for $\sigma_{k;n}(s)$ is equivalent to
\begin{equation}\label{eq:ODE_system_etas_finite_n}
\begin{aligned}
&\eta_{k;n}'(s)
=\frac{k}{2}\sum_{j=1}^{k-1}
\exp \left(\frac{j(k-j)}{n} s\right)
\eta_{j;n}(s) \eta_{k-j;n}(s), \qquad k\in \mathbb N, \qquad s\in \mathbb R,
\\
&\eta_{0;n}(s) \equiv  1 \quad (\text{for all } s\in \mathbb R), \qquad \eta_{k; n}(0) = 1 \quad (\text{for all } k\in  \mathbb N).
\end{aligned}
\end{equation}
Similarly, the ODE system~\eqref{eq:ODE_system_infinite_n} is equivalent to
\begin{equation}\label{eq:ODE_system_etas_infinite_n}
\begin{aligned}
&\eta_{k;\infty}'(s)
=\frac{k}{2}\sum_{j=1}^{k-1}
\eta_{j;\infty}(s) \eta_{k-j;\infty}(s), \qquad k\in \mathbb N,\qquad s\in \mathbb R,
\\
&\eta_{0;\infty}(s) \equiv  1 \quad (\text{for all } s\in \mathbb R), \qquad \eta_{k; }(0) = 1 \quad (\text{for all } k\in  \mathbb N).
\end{aligned}
\end{equation}

\subsection{Convergence to the inviscid limit.}
Our aim is now to prove that for every $k\in \mathbb N$ and $A>0$, we have
$$
\lim_{n\to\infty}\sigma_{k;n}(s) = \sigma_{k;\infty}(s) \qquad \text{ uniformly in } s\in [-A,A].
$$
By~\eqref{eq:etas_def}, it suffices to prove a similar claim for $\eta_{k;n}(s)$ and $\eta_{k;\infty}(s)$.  This will be done in the next
\begin{proposition}[Uniform convergence for the triangular ODE system] \label{prop:eta_converge_hermite}
 For each $n\in\mathbb{N}$, let
$\{\eta_{m;n}\}_{m=1}^\infty\subseteq C^\infty(\mathbb R)$ solve the ODE system~\eqref{eq:ODE_system_etas_finite_n}.
Let also $\{\eta_{m;\infty}\}_{m=1}^\infty\subseteq C^\infty(\mathbb R)$ solve~\eqref{eq:ODE_system_etas_infinite_n}.
Fix $A>0$.
Then, for every $k\in \mathbb N$ there exists a constant $C_{k,A}<\infty$ (depending only on $k$ and $A$)  such that
\[
\sup_{s\in [-A,A]} |\eta_{k;n}(s)-\eta_{k;\infty}(s)|  \le  \frac{C_{k,A}}{n}
\qquad\text{for all }n\in\mathbb{N}.
\]
In particular, for every $k\in \mathbb N$ we have $\lim_{n\to\infty}\eta_{k;n}(s) = \eta_{k;\infty}(s)$ uniformly in $s\in [-A,A]$.
\end{proposition}

\begin{proof}
We argue by induction on $k$.  Fix $A>0$ and let  $\|f\|_\infty:= \sup_{x\in [-A,A]} |f(x)|$.

\smallskip
\noindent\textit{Step 1: Uniform bounds.}
We claim that there exist finite constants $B_1,B_2,\ldots$ such that, for all $m\in \mathbb N$,
\begin{equation}\label{eq:uniform_bounds_ODE_system}
\sup_{n\in\mathbb{N}}\ \|\eta_{m;n}\|_\infty\le B_m,
\qquad
\|\eta_{m;\infty}\|_\infty\le B_m.
\end{equation}

For $m=1$, we have $\eta_{1;n}\equiv\eta_{1;\infty}\equiv 1$ since the sums in~\eqref{eq:ODE_system_etas_finite_n} and~\eqref{eq:ODE_system_etas_infinite_n} are empty. Thus,  we can take $B_1=1$.
Assume that~\eqref{eq:uniform_bounds_ODE_system} holds for all indices $m < k$. To prove that it holds for $m=k$, we integrate~\eqref{eq:ODE_system_etas_finite_n} and use $|s|\le A$:
\[
\|\eta_{k;n}\|_\infty
=
\sup_{s\in [-A,A]}
|\eta_{k;n}(s)|
\le 1 + A \cdot \frac{k}{2}\sum_{j=1}^{k-1} \eee^{ j(k-j)A} B_j B_{k-j}.
\]
From~\eqref{eq:ODE_system_etas_infinite_n} we similarly obtain
\[
\|\eta_{k;\infty}\|_\infty
=
\sup_{s\in [-A,A]}|\eta_{k;\infty}(s)|
\le 1 + A \cdot \frac{k}{2}\sum_{j=1}^{k-1} B_j B_{k-j}.
\]
Define $B_k$ as the maximum of the two right-hand sides. This proves the uniform bounds~\eqref{eq:uniform_bounds_ODE_system} for $m=k$ and completes the induction.

\smallskip
\noindent\textit{Step 2: Convergence with rate.}
Define $\delta_{m;n}(s):=\eta_{m;n}(s)-\eta_{m;\infty}(s)$ for $m\in \N, n\in \N$. We claim: For every $m\in \N$ there exists a finite constant $C_{m,A}$ such that $\|\delta_{m;n}\|_{\infty}\le C_{m,A}/n$ for all $n\in \N$.

We use induction on $m$.  The case $m=1$ holds since $\eta_{1;n}\equiv\eta_{1;\infty}\equiv 1$. Assume the claim holds for all indices $m<k$.
To prove it for $m=k$, subtract the integral forms of \eqref{eq:ODE_system_etas_finite_n} and~\eqref{eq:ODE_system_etas_infinite_n} to get
\[
\delta_{k;n}(s)
=\frac{k}{2}\sum_{j=1}^{k-1}\int_0^s
\left(
\eee^{\frac{j(k-j)}{n}\tau} \eta_{j;n}(\tau) \eta_{k-j;n}(\tau)
-\eta_{j;\infty}(\tau) \eta_{k-j;\infty}(\tau)
\right) \dd\tau,
\qquad
s\in \R.
\]
Split the integrand into $I(\tau)+J(\tau)$ with
\[
I(\tau):= \left(\eee^{\frac{j(k-j)}{n}\tau}-1\right)\eta_{j;n}(\tau)\eta_{k-j;n}(\tau),
\qquad
J(\tau) := \eta_{j;n}(\tau)\eta_{k-j;n}(\tau)-\eta_{j;\infty}(\tau)\eta_{k-j;\infty}(\tau).
\]
In the following, $C', C'',\ldots$ denote finite constants that depend only on $k$ and $A$.

\smallskip
\noindent\textit{Step 2A.}
To bound $I(\tau)$,  we observe that
\[
\max_{j=1,\ldots, k-1} \sup_{\tau \in [-A,A]} \left|\eee^{\frac{j(k-j)}{n}\tau}-1\right|
\le \frac{C'}n.
\]
Together with Step~1 this gives
\[
\left|\int_0^s I(\tau)  \dd \tau\right|
\le \frac{C'}{n} A B_jB_{k-j}, \qquad s\in [-A,A].
\]

\smallskip
\noindent\textit{Step 2B.}
To bound $J(\tau)$, we write
\[
\eta_{j;n}(s)\eta_{k-j;n}(s)-\eta_{j;\infty}(s)\eta_{k-j;\infty}(s)
= \delta_{j;n}(s) \eta_{k-j;\infty}(s)+\eta_{j;n}(s) \delta_{k-j;n}(s)+\delta_{j;n}(s) \delta_{k-j;n}(s).
\]
Applying Step~1 gives, for all $s\in [-A,A]$,
\begin{align*}
\left|\int_0^s J(\tau) \dd\tau\right|
&\leq
A \|\eta_{j;n}\eta_{k-j;n}-\eta_{j;\infty}\eta_{k-j;\infty} \|_\infty \\
&\leq
C''
\left(\|\delta_{j;n}\|_{\infty}+\|\delta_{k-j;n}\|_{\infty}
+\|\delta_{j;n}\|_{\infty} \|\delta_{k-j;n}\|_{\infty}\right).
\end{align*}
For $j=1,\ldots, k-1$, the induction assumption gives $\|\delta_{j;n}\|_{\infty}\le C_{j,A}/n$ and $\|\delta_{k-j;n}\|_{\infty}\le C_{k-j,A}/n$.
Altogether, this implies
$$
\left|\int_0^s J(\tau) \dd\tau\right| \leq \frac{C'''}n.
$$

Taken together, Steps 2A and 2B imply the existence of a finite constant $C_{k,A}$ such that $\|\delta_{k;n}\|_{\infty}\le C_{k,A}/n$ for all $n\in \N$. This completes the induction.
\end{proof}

\subsection{Identification of the inviscid limit}\label{subsec:identification}
We showed that for every $k\in \N$ and $s\in \R$,
\begin{equation}
\sg_{k;n}(s) = \frac1n\sum_{j=1}^n x_{j;n}(s)^k  \overset{}{\underset{n\to\infty}\longrightarrow} \sigma_{k;\infty}(s) = \eee^{ks} \eta_{k;\infty}(s),
\end{equation}
where the functions $\eta_{k;\infty}(s)$, $k\in \N$, form the \emph{unique} solution to the ODE system
\begin{equation}\label{eq:ODE_system_infinite_n_repeat}
\begin{aligned}
&\eta_{k;\infty}'(s)
=\frac{k}{2}\sum_{j=1}^{k-1}\eta_{j;\infty}(s)\eta_{k-j;\infty}(s), \qquad k\in \mathbb N, \qquad s\in \mathbb R,
\\
&\eta_{k; \infty}(0) = 1 \quad (\text{for all } k\in  \mathbb N).
\end{aligned}
\end{equation}
The next proposition gives a different way of writing the solution to the same system. It is known~\cite{biane}, but we provide a proof to keep the argument self-contained.

\begin{proposition}\label{prop:identification_limit_hermite}
Let $m_k(s)=\frac{1}{k} [u^{k-1}] ((1+u)^k \eee^{sku})$, $k\in \N$, $s\in \R$. Then,
\begin{equation}\label{eq:ODE_system_for_moments}
\begin{aligned}
&m_{k}'(s)
=\frac{k}{2}\sum_{j=1}^{k-1}m_{j}(s)m_{k-j}(s), \qquad k\in \mathbb N, \qquad s\in \mathbb R,
\\
&m_{k}(0) = 1 \quad (\text{for all } k\in  \mathbb N).
\end{aligned}
\end{equation}
Consequently, $\eta_{k;\infty}(s) = m_k(s)$ and $\sigma_{k;\infty}(s) = \eee^{ks} m_k(s) =  \frac{\eee^{sk}}{k} [u^{ k-1}]((1+u)^{k} \eee^{ sk u})$ for all $s\in \R$ and $k\in \N$.
\end{proposition}
\begin{proof}
Consider the generating function
$$
\psi(z;s):=\sum_{k=1}^\infty m_k(s) z^k.
$$
Write $m_k(s)= \frac{1}{k} [u^{k-1}] (\phi(u))^k$, $k\in \N$,  where $\phi(u) = (1+u)\eee^{su}$. From the Lagrange inversion formula (see~\cite[Equation~(2.1.1)]{gessel_lagrange_inversion})  it follows that $\psi(z;s)$ solves $\psi(z;s) = z \phi(\psi(z;s))$, that is
$$
\psi(z;s) -  z (1+\psi(z;s)) \eee^{s \psi(z;s)} = 0.
$$
Writing $F(z,s,\psi) = \psi -  z (1+\psi) \eee^{s \psi}$, we have $F(z,s,\psi(z;s)) \equiv 0$.  Treating $\psi$ as an independent variable, the partial derivatives of $F$ in $z$ and $s$ are
$$
\partial_z F(z,s,\psi) = -(1+\psi)\eee^{s \psi},
\qquad
\partial_s F(z,s,\psi) = - z (1+\psi) \psi\eee^{s \psi} = z\psi \cdot \partial_z F(z,s,\psi).
$$
Taking the derivative of the equation $F(z,s,\psi(z;s)) \equiv 0$ in  $z$ and $s$ we get
\begin{align*}
&0 = \partial_z (F(z,s,\psi(z;s))) = (\partial_z F)(z,s,\psi(z;s))  +  (\partial_\psi F) (z,s,\psi(z;s)) \partial_z \psi(z;s),
\\
&0 = \partial_s (F(z,s,\psi(z;s))) = (\partial_s F)(z,s,\psi(z;s))  +  (\partial_\psi F) (z,s,\psi(z;s)) \partial_s \psi(z;s).
\end{align*}
Comparing these equations, it follows that
$$
(\partial_z F)(z,s,\psi(z;s)) \cdot \partial_s \psi(z;s) = (\partial_s F)(z,s,\psi(z;s)) \cdot \partial_z \psi(z;s).
$$
Recalling that $\partial_s F(z,s,\psi) = z\psi \cdot \partial_z F(z,s,\psi)$ gives the PDE
$$
\partial_s \psi(z;s)  =  z \psi(z;s)\cdot \partial_z \psi(z;s) = \frac 12 z\partial_z \psi^2(z;s).
$$
Recall that $\psi(z;s)=\sum_{k=1}^\infty m_k(s) z^k$ and expand
$$
\psi^2(z;s)=\sum_{k = 1}^\infty z^k \sum_{j=1}^{k-1} m_{j}(s) m_{k-j}(s),
\qquad
\frac 12 z\partial_z \psi^2(z;s) = \sum_{k = 1}^\infty \frac k2 z^k \sum_{j=1}^{k-1} m_{j}(s) m_{k-j}(s),
$$
where we used that $z\partial_z z^{k}=kz^k$.  Coefficient extraction gives the ODE system~\eqref{eq:ODE_system_for_moments}.
\end{proof}

\begin{remark}
Let us sketch a different approach to the proof of $\sigma_{k; \infty} (s) =  \frac{\eee^{sk}}{k} [u^{ k-1}] ((1+u)^{k} \eee^{ sk u})$. Introduce the generating function $\Sigma_\infty (x;s) := \sum_{k=0}^\infty\sigma_{k;\infty}(s) x^{-k}$. Repeating backwards the argument of Section~\ref{subsec:finite_n_ODE} one shows that $\Sigma_\infty (x;s)$
satisfies the inviscid Burgers-type PDE
$$
\partial_s \Sigma_{\infty}(x;s)  =  -\tfrac12 L (\Sigma^2_\infty(x;s)) = -x \Sigma_\infty(x;s) \partial_x\Sigma_\infty(x;s),\qquad L=x\partial_x.
$$
The initial condition is $\Sigma_{\infty}(x;0) = x/(x-1)$ and comes from $\sigma_{k;\infty}(0) =1$, $k\in \N_0$. Solving this first-order quasilinear transport equation using the method of characteristics gives the implicit solution
$$
x = \frac{\Sigma_\infty(x;s)}{\Sigma_\infty(x;s)-1} \eee^{s \Sigma_\infty(x;s)}.
$$
Lagrange inversion formula (see~\cite[Equation~(2.1.1)]{gessel_lagrange_inversion}) gives the claimed formula for $\sigma_{k; \infty} (s)$.
\end{remark}

\subsection{$S$-transform}
For completeness, we provide a proof of the following fact known from~\cite[p.~4]{biane}.
\begin{proposition}\label{prop:S_transform_hermite}
Consider a probability distribution $\mu^{(s)}$ with the $S$-transform $S(z;s) =\eee^{-s(z + \frac 12)}$, where $s\in \R$ is a parameter. Then, the $k$-th moment of $\mu^{(s)}$ equals $\frac{\eee^{sk/2}}{k} [u^{k-1}] ((1+u)^k \eee^{sku})$, for all $k\in \N$.
\end{proposition}
\begin{proof}
Let $\psi(z)=\sum_{k=1}^\infty (\int_\C u^k \mu^{(s)}(\dd u)) z^k$ be the $\psi$-transform of $\mu^{(s)}$. The definition of the $S$-transform, Equation~\eqref{eq:S_transform_def},  gives
$$
\frac{1+\psi(t)}{\psi(t)} t = S\left(\psi(t); s\right) = \eee^{-s(\psi(t)+\frac12)}.
$$
Hence, $\psi(t) = t \phi(\psi(t))$ with $\phi(y) = (1+y)\eee^{s(y+\frac12)}$.  By the Lagrange inversion formula (see~\cite[Equation~(2.1.1)]{gessel_lagrange_inversion}),  $(\int_\C u^k \mu^{(s)}(\dd u))=  [z^{k}] \psi(z) =   \frac{1}{k} [u^{k-1}] (\phi(u))^k$, $k\in \N$, and the proof is complete.
\end{proof}

\subsection{Weak convergence}
We are now ready to complete the proof of Theorem~\ref{theo:multiplicative_hermite_polys_zeros}. Let $k\in \N$.
In the previous sections we showed that the $k$-th moment of $\llbracket  H_n^*(\cdot; s/n)\rrbracket_n$ converges to the $k$-th moment of a probability measure $\mu^{(s)}$ with the $S$-transform $S(z;s) =\eee^{-s(z + \frac 12)}$. Indeed, the $k$-th moment of $\llbracket  H_n^*(\cdot; s/n)\rrbracket_n$ is $\eee^{-sk/2} \sigma_{k;n}(s)$, see Section~\ref{subsec:setup_hermite}. By   Propositions~\ref{prop:eta_converge_hermite} and~\ref{prop:identification_limit_hermite},  $\eee^{-sk/2} \sigma_{k;n}(s)$ converges to $\frac{\eee^{sk/2}}{k} [u^{ k-1}] ((1+u)^{k} \eee^{ sk u})$,  as $n\to\infty$. By Proposition~\ref{prop:S_transform_hermite}, the limit coincides with the $k$-th moment of the distribution $\mu^{(s)}$ with the $S$-transform $S(z;s) =\eee^{-s(z + \frac 12)}$.

It remains to show that convergence of moments implies weak convergence of $\llbracket  H_n^*(\cdot; s/n)\rrbracket_n$ to $\mu^{(s)}$.

If $s\leq 0$, we know that the distributions $\llbracket  H_n^*(\cdot; s/n)\rrbracket_n$ and $\mu^{(s)}$ are concentrated  on the unit circle.
%and invariant under the complex conjugation map $z\mapsto \bar z = 1/z$.
Convergence of moments of any order $k\in \N_0$ (and hence, by complex conjugation, $k\in \Z$) implies weak convergence by Weyl's criterion.

If $s\geq 0$, we know that  $\llbracket  H_n^*(\cdot; s/n)\rrbracket_n$ and $\mu^{(s)}$ are concentrated on $(0,\infty)$.  The measure $\mu^{(s)}$ is compactly supported and hence uniquely determined by its moments. Altogether this implies weak convergence by the standard method of moments; see~\cite[Section~3.3.5]{durrett_book}.

By the identity $z^n H_n^*(1/z;s) = (-1)^n H_n^*(z;s)$, the measures $\llbracket  H_n^*(\cdot; s/n)\rrbracket_n$ are invariant w.r.t.\ $z\mapsto 1/z$. Hence, $\mu^{(s)}$ is also invariant under $z\mapsto 1/z$.

\section{Proof of Theorem~\ref{theo:multiplicative_laguerre_polys_zeros}}\label{sec:proof_laguerre}
\subsection{Setup}
The basic idea  is the same as in the proof of Theorem~\ref{theo:multiplicative_hermite_polys_zeros}, but the technical details are different and more involved.
Fix an integer $n\in \N$ and a constant $b\in\C$. We shall use the notation
$$
P_n(x;s) := L_n^*(x; b,s) = (x\partial_x + b)^s (x-1)^n, \qquad  x\in \C,  \;  s\in \N_0.
$$
The ``time'' parameter $s\in \N_0$ is now discrete. The role of the heat PDE is now played by the identity
\begin{equation}\label{eq:evolution}
P_n(x;s+1)  =  \left(x \partial_x + b\right) P_n(x;s), \qquad x\in \C, \;  s\in \N_0.
\end{equation}
Let  $x_{1;n}(s)\in \C,\dots,x_{n;n}(s)\in \C$ be the zeros of the polynomial $\Pn(\cdot;s)$ (counted with multiplicity), so that
\begin{equation}\label{eq:root-factorization-n-laguerre}
\llbracket  P_n(\cdot; s)\rrbracket_n = \frac 1n \sum_{j=1}^n  \delta_{x_{j;n}(s)}.
\end{equation}
We denote the $k$-th moment of  $\llbracket  P_n(\cdot; s)\rrbracket_n$ by
\begin{equation}\label{eq:sigma-def-n-laguerre}
\sg_{k;n}(s) := \frac1n\sum_{j=1}^n x_{j;n}(s)^k,
\quad k\in \N_0, \;  s\in \N_0.
\end{equation}
Note that $\sg_{0;n}(s)= 1$ for all $s\in \N_0$, and $\sigma_{k; n}(0) = 1$ for all $k\in  \mathbb N$. Define the generating function
\begin{equation}\label{eq:Sigma-series}
\Sigma_n(x;s) := \sum_{k=0}^\infty\sg_{k;n}(s) x^{-k}
 = \frac xn\sum_{j=1}^n\frac{1}{x-x_{j;n}(s)}
 = \frac{x}{n} \frac{\partial_x P_n(x;s)}{P_n(x;s)}.
\end{equation}
It is well-defined and analytic for sufficiently large $|x|$. Note that $\Sigma_n(\infty;s):=\lim_{|x|\to\infty}\Sigma_n(x;s) = 1$.

\subsection{Difference equation}
We claim that (for sufficiently large $|x|$)
\begin{equation}\label{eq:Sigma-update}
\Sigma_n(x;s+1)-\Sigma_n(x;s)=\frac{x \partial_x \Sigma_n(x;s)}{ b+n \Sigma_n(x;s)},
\qquad
s\in \N_0.
\end{equation}
\begin{proof}
Differentiate \eqref{eq:evolution} with respect to $x$:
\begin{align*}
\partial_x P_n(x;s+1)
= \partial_x \left(x \partial_x P_n(x;s) + b P_n(x;s)\right)
= (1+b) \partial_x P_n(x;s) + x \partial_x^2 P_n(x;s).
\end{align*}
Therefore,
\begin{equation}\label{eq:Sigma-s+1}
\Sigma_n(x;s+1)
=
\frac{x}{n} \frac{\partial_x P_n(x;s+1)}{P_n(x;s+1)}
= \frac{x}{n}
\frac{(1+b) \partial_x P_n(x;s) + x \partial_x^2 P_n(x;s)}
{x \partial_x P_n(x;s) + b P_n(x;s)}.
\end{equation}
In the following, $P_n$ stands for $P_n(x;s)$. Subtracting \eqref{eq:Sigma-series} gives
\begin{align*}
\Sigma_n(x;s+1) - \Sigma_n(x;s)
&= \frac{x}{n}\left(
\frac{(1+b) \partial_x P_n + x \partial_x^2 P_n}{x \partial_x P_n + b P_n}
- \frac{\partial_x P_n}{P_n}\right)  \\
&= \frac{x}{n}
\frac{\left((1+b) \partial_x P_n + x \partial_x^2 P_n\right) P_n
- \partial_x P_n\left(x \partial_x P_n + b P_n\right)}
{\left(x \partial_x P_n + b P_n\right) P_n} \\
&= \frac{x}{n}
\frac{\partial_x P_n\cdot P_n + x\left(\partial_x^2 P_n\cdot P_n - (\partial_x P_n)^2\right)}
{\left(x \partial_x P_n + b P_n\right) P_n}.
\end{align*}
We are going to represent the right-hand side in terms of $\Sigma_n(x;s)$.
By~\eqref{eq:Sigma-series},
\begin{equation*}
\frac{\partial_x P_n(x;s)}{P_n(x;s)} = \frac{n}{x} \Sigma_n(x;s),
\qquad
\partial_x \left(\frac{\partial_x P_n}{P_n}\right)
= \partial_x \left(\frac{n}{x} \Sigma_n(x;s)\right)
= n\left(\frac{\partial_x \Sigma_n(x;s)}{x} - \frac{\Sigma_n(x;s)}{x^2}\right).
\end{equation*}
Using the identity
\begin{equation*}
\partial_x \left(\frac{\partial_x P_n}{P_n}\right)
= \frac{\partial_x^2 P_n}{P_n} - \left(\frac{\partial_x P_n}{P_n}\right)^2,
\end{equation*}
we can rewrite the numerator as
\begin{align*}
\partial_x P_n\cdot P_n
+ x(\partial_x^2 P_n\cdot P_n - (\partial_x P_n)^2)
&= P_n^2  \cdot \left(\frac{\partial_x P_n}{P_n}
+ x \partial_x \left(\frac{\partial_x P_n}{P_n}\right)\right)
\\
&=
P_n^2 \cdot \left(\frac{n}{x} \Sigma_n(x;s)
+ x \cdot n\left(\frac{\partial_x \Sigma_n(x;s)}{x} - \frac{\Sigma_n(x;s)}{x^2}\right)\right)
\\
&=
P_n^2 \cdot  n \partial_x \Sigma_n(x;s).
\end{align*}
For the denominator we note
\begin{equation*}
\left(x \partial_x P_n + b P_n\right) P_n
= P_n^2 \cdot\left(x \frac{\partial_x P_n}{P_n} + b\right)
= P_n^2 \cdot \left(n \Sigma_n(x;s) + b\right).
\end{equation*}
Putting everything together gives
\begin{align*}
\Sigma_n(x;s+1) - \Sigma_n(x;s)
= \frac{x}{n}\cdot
\frac{P_n^2 \cdot n \partial_x \Sigma_n(x;s)}{P_n^2\cdot (n \Sigma_n(x;s)+b)}
= \frac{x \partial_x \Sigma_n(x;s)}{b + n \Sigma_n(x;s)}.
\end{align*}
\end{proof}

\subsection{Bell polynomials}
The next step is to derive a system of difference equations for $\sigma_{k;n}(s)$. As it turns out, these equations involve Bell polynomials.
In this section, we recall some basic properties of ordinary Bell polynomials.

For indeterminates $x_1,x_2,\dots$, the \emph{partial ordinary Bell polynomials} are defined by the generating function
\begin{equation}\label{eq:bell_poly_ordinary_gen_function}
\sum_{n=k}^\infty \widehat{B}_{n,k}(x_1,\ldots, x_{n-k+1}) t^{n}
=\left(\sum_{m=1}^\infty x_m t^{m}\right)^{k}.
\end{equation}
By convention,  $\widehat B_{m,0} = 0$ for $m\in \N$ and $\widehat B_{0,0} = 1$. Coefficient extraction gives
\[
\widehat{B}_{n,k}(x_1,\ldots, x_{n-k+1})
= [t^n]\left(\sum_{m=1}^\infty x_m t^m\right)^{k}
= \sum_{\substack{n_1+\cdots+n_k=n\\ n_i\in \N}}
  x_{n_1}\ldots x_{n_k}.
\]
In particular, $\widehat{B}_{n,k}(x_1,\ldots, x_{n-k+1})$ indeed depends only on $x_1,\ldots, x_{n-k+1}$. \emph{Weighted complete ordinary Bell polynomials} are defined by
\[
\widehat B_{m}(x_1,\dots, x_m; y)
:=
\sum_{r=0}^{m} y^{r} \widehat B_{m,r}\left(x_1,\dots,x_{m-r+1}\right)
=
1 + \sum_{r=1}^{m} y^{r} \widehat B_{m,r}\left(x_1,\dots,x_{m-r+1}\right),
\qquad m\in \N.
\]
Note that $\widehat{B}_{m}(x_1,\ldots, x_m;y)$ indeed depends only on $x_1,\ldots, x_{m},y$. By convention, $\widehat B_{0} := 1$.

\subsection{Difference equations for the moments}
We claim that for all $k\in \N_0$, $n\in \N_0$, $s\in \N_0$ (and assuming $b\neq -n$),
\begin{align}
\sigma_{k;n}(s+1)-\sigma_{k;n}(s)
&= -\frac{1}{b+n}\sum_{j=1}^{k} j \sigma_{j;n}(s)
\widehat B_{ k-j} \left(\sigma_{1;n}(s),\ldots,\sigma_{k-j;n}(s); -\tfrac{n}{b+n}\right)
\label{eq:sigmadiff-final-0}\\
&=
 -\frac{k \sigma_{k;n}(s)}{b+n}
-\frac{1}{b+n}\sum_{j=1}^{k-1} j \sigma_{j;n}(s)
\widehat B_{ k-j} \left(\sigma_{1;n}(s),\ldots,\sigma_{k-j;n}(s); -\tfrac{n}{b+n}\right). \label{eq:sigmadiff-final}
\end{align}

\begin{proof}
The  Laurent series $\Sigma_n(x;s)=\sum_{k=0}^{\infty}\sigma_{k;n}(s) x^{-k}$ and~\eqref{eq:Sigma-update} give
$$
\sigma_{k;n}(s+1)-\sigma_{k;n}(s) = [x^{-k}] (\Sigma_n(x;s+1)-\Sigma_n(x;s)) = [x^{-k}] \frac{x \partial_x \Sigma_n(x;s)}{ b+n \Sigma_n(x;s)}.
$$
For the numerator we have
\begin{equation}\label{eq:xDxSigma}
x \partial_x\Sigma_n(x;s)=-\sum_{j=1}^{\infty} j \sigma_{j;n}(s) x^{-j}.
\end{equation}
We expand the reciprocal of the denominator into a geometric series in the variable $\Sigma_n(x;s)-1$:
\begin{equation}\label{eq:geom}
\frac{1}{b+n \Sigma_n(x;s)}
=\frac{1}{b+n+n (\Sigma_n(x;s)-1)}
=\frac{1}{b+n}\sum_{r=0}^{\infty}\left(-\frac{n}{b+n}\right)^{ r} (\Sigma_n(x;s)-1)^r.
\end{equation}
Multiplying \eqref{eq:xDxSigma} and \eqref{eq:geom} and comparing the coefficient of $x^{-k}$ gives
\begin{equation}\label{eq:sigmadiff-preBell}
\sigma_{k;n}(s+1)-\sigma_{k;n}(s)
=\sum_{j=1}^{k}\left(-j \sigma_{j;n}(s)\right)\cdot \frac{1}{b+n}
\sum_{r=0}^{k-j}\left(-\frac{n}{b+n}\right)^{ r} C_r(k-j;s),
\end{equation}
where
$$
C_r(\ell;s) = [x^{-\ell}] \left((\Sigma_n(x;s)-1)^r\right)
=
[x^{-\ell}] \left(\sum_{k=1}^{\infty}\sigma_{k;n}(s) x^{-k} \right)^r
=
\widehat B_{\ell,r}\left(\sigma_{1;n}(s),\sigma_{2;n}(s),\ldots\right),
$$
the last step being a consequence of~\eqref{eq:bell_poly_ordinary_gen_function}.
Therefore \eqref{eq:sigmadiff-preBell} becomes
\begin{align*}
\sigma_{k;n}(s+1)-\sigma_{k;n}(s)
&=
 - \frac{1}{b+n}\sum_{j=1}^{k} j \sigma_{j;n}(s)
\sum_{r=0}^{k-j}\left(-\frac{n}{b+n}\right)^{ r} \widehat B_{k-j,r}\left(\sigma_{1;n}(s), \sigma_{2;n}(s),\ldots\right)
\\
&=
-\frac{1}{b+n}\sum_{j=1}^{k} j \sigma_{j;n}(s)
\widehat B_{ k-j} \left(\sigma_{1;n}(s),\sigma_{2;n}(s),\dots; -\tfrac{n}{b+n}\right),
\end{align*}
which proves~\eqref{eq:sigmadiff-final-0}. Separating the term with $j=k$ and using $\widehat B_0 = 1$ gives~\eqref{eq:sigmadiff-final}.
\end{proof}

\subsection{Forced system of difference equations}
Write the difference equations~\eqref{eq:sigmadiff-final}  as
\begin{equation}\label{eq:sigmadiff-final_split_first_term}
\sigma_{k;n}(s+1)-\sigma_{k;n}(s)
=
- \frac{k  \sigma_{k;n}(s)}{b+n}
+
\frac 1n \Phi_k \left(\sigma_{1;n}(s), \ldots, \sigma_{k-1;n}(s); -\frac{n}{b+n}\right),
\end{equation}
where
\begin{equation}\label{eq:Phi_def_Bell_poly}
\Phi_k (x_1, \ldots, x_{k-1}; a)
:=
 a   \sum_{j=1}^{k-1} j  x_j
\widehat{B}_{ k-j} \left(x_1,  x_2,\ldots, x_{k-j}; a\right).
\end{equation}
To remove the first term on the right-hand side of~\eqref{eq:sigmadiff-final_split_first_term}, write
\[
c_{j;n}:=1-\frac{j}{ b+n },\qquad
\sigma_{j;n}(s)=c_{j;n}^{ s} \eta_{j;n}(s),
\qquad
j\in \N_0, \; n\in \N_0, \; s\in \N_0.
\]
In terms of the new functions  $\eta_{k;n}(s)$, the system of difference equations takes the form
\begin{equation}\label{eq:eta_system_difference_equations_laguerre}
\begin{aligned}
&\eta_{k;n}(s+1)-\eta_{k;n}(s)
=
\frac{1}{n c_{k;n}^{ s+1}}
\Phi_k \left(
c_{1;n}^{ s}\eta_{1;n}(s),\ldots,c_{k-1;n}^{ s}\eta_{k-1;n}(s);
-\frac{n}{b+n}
\right),
\qquad
k\in \N,  \;  s\in \N_0,
\\
&\eta_{0;n}(s) =  1 \quad (\text{for all } s\in \N_0), \qquad \eta_{k; n}(0) = 1 \quad (\text{for all } k\in  \mathbb N).
\end{aligned}
\end{equation}
Note that $\Phi_k (x_1, \ldots, x_{k-1}; a)$ is a polynomial in all of its arguments, including $a$.

\subsection{Convergence to an ODE limit as $n\to\infty$ for the forced system}
From now on, we assume that $b=b_n\in \C$ depends on $n$ in such a way  that
$$
\lim_{n\to\infty} \frac{b_n}{n} = \beta\in\C, \qquad \beta \neq -1.
$$
The next result states that, as $n\to\infty$, the function $\eta_{k;n}(s)$ converges to a $C^\infty$-function $g_k(\gamma)$, with the time scaling $\gamma = s/n$, and that the functions $g_k$ satisfy a triangular system of ODE's which is a natural continuum limit of~\eqref{eq:eta_system_difference_equations_laguerre}.
Define
\[
\alpha_n:= -\frac{n}{b_n+n} \overset{}{\underset{n\to\infty}\longrightarrow} -\frac{1}{1+\beta} =: \alpha.
\]

\begin{proposition}\label{prop:conv_to_ODE_laguerre}
For every fixed $k\in \N_0$ there exists a unique $C^\infty$-function $g_k:[0,\infty) \to \mathbb{C}$ such that
\begin{equation}\label{eq:prop:conv_to_ODE_laguerre}
\lim_{n\to\infty} \sup_{s\in [0, An]\cap \Z} \left|\eta_{k;n}(s)-g_k(s/n)\right| = 0,
\end{equation}
for every $A>0$.
Moreover,  the functions  $g_k(\gamma)$, $k\in \N_0$,  are  uniquely characterized by the triangular ODE system
\begin{equation}\label{eq:fk-ode}
\begin{aligned}
g_k'(\gamma)
&=
\Phi_k \left(g_1(\gamma), \ldots, g_{k-1}(\gamma); \alpha\right)
=
\alpha \sum_{j=1}^{k-1} j  g_j(\gamma)
\widehat{B}_{ k-j} \left(g_1(\gamma), \ldots, g_{k-j}(\gamma); \alpha\right),
\quad k\in \N,\ \gamma\geq 0,
\\
g_{0}(\gamma) &=  1 \quad (\text{for all } \gamma\geq 0), \qquad g_{k}(0) = 1 \quad (\text{for all } k\in  \mathbb N).
\end{aligned}
\end{equation}
\end{proposition}
\begin{remark}
The ODE system~\eqref{eq:fk-ode} is triangular: each $g_k'$ is a polynomial of $g_0,\ldots,g_{k-1}$ and $\alpha$. The system can be solved successively starting with $g_0\equiv 1$. By induction, one shows that each $g_k(\gamma)$ is a polynomial in $\gamma$ and $\alpha$. For example,
\[
\begin{aligned}
g_1(\gamma)&=1,\\
g_2(\gamma)&=1+\alpha^{2}\gamma,\\
g_3(\gamma)&=1+\left(3\alpha^{2}+\alpha^{3}\right)\gamma+\tfrac{3}{2}\alpha^{4}\gamma^{2}.
\end{aligned}
\]
\end{remark}
\begin{proof}
Recall that $c_{j;n} = 1-\frac{j}{b_n+n}$.  For each $j\in \N_0$ and $A>0$ we have
$$
\lim_{n\to\infty} \sup_{s\in [0, An]} |c_{j;n}^{ s} - \eee^{-sj/(b_n+n)}| = 0.
$$
In particular, there are constants $1<L_{j,A}<\infty$ such that for all $n$ and all $0\le s\le An$,
\[
0 <L_{j,A}^{-1}\le |c_{j;n}^{s}|\le L_{j, A} < \infty.
\]

\smallskip
\noindent \emph{Uniform boundedness.}
Fix $A>0$ and $k\in\mathbb{N}_0$.  We claim that
\[
\sup_{n\in\mathbb{N}} \sup_{s\in [0, An]\cap \Z}\ |\eta_{k;n}(s)|\ <\ \infty .
\]
We proceed by induction on $k$. The base case $k=0$ is trivial, since $\eta_{0;n}\equiv 1$.  For the inductive step, assume the claim holds for $j\in \{0,\dots,k-1\}$, i.e., there exist $M_{j,A}$ such that
\[
\sup_{n\in \N} \sup_{s\in [0, An]\cap \Z}\ |\eta_{j;n}(s)|\ \le\ M_{j,A}, \qquad j\in \{0,\ldots, k-1\}.
\]
Since $|c_{j;n}^{s}|\le L_{j,A}$, we also have
\[
\sup_{n\in \N} \sup_{s\in [0, An]\cap \Z} |c_{j;n}^{ s}\eta_{j;n}(s)| \le  L_{j,A}  M_{j,A},\qquad j\in \{0,,\dots,k-1\}.
\]
Because $\Phi_k$ is a polynomial, there exists a constant $B_{k, A} < \infty$ with
\[
\sup_{n\in \N} \sup_{s\in [0, An]\cap \Z} \left|\Phi_k\left(c_{1;n}^{ s}\eta_{1;n}(s),\ldots,c_{k-1;n}^{ s}\eta_{k-1;n}(s);\alpha_n\right)\right|
\le B_{k, A}
\]
(also using $\alpha_n\to \alpha$, which implies $\sup_{n\in \N} |\alpha_n|<\infty$).  Hence, by~\eqref{eq:eta_system_difference_equations_laguerre}, there exists $B_{k, A}' < \infty$ such that
\[
\sup_{s\in [0, An]\cap \Z} |\eta_{k;n}(s+1)-\eta_{k;n}(s)|
\le \frac{B_{k,A}'}{n}, \qquad n\in \N.
\]
Summing from $0$ to $s-1$ and using $\eta_{k;n}(0)=1$ gives
\[
|\eta_{k;n}(s)|
\le 1 + \sum_{r=0}^{s-1}\frac{B_{k,A}'}{n}
\le 1 + B_{k,A}' \frac{s}{n}
\le 1 + B_{k,A}' A,
\]
for all $n\in \N$ and $s\in [0, An]\cap \Z$. This completes the induction.

\smallskip
\noindent
\emph{Induction assumption and notation.} To prove~\eqref{eq:prop:conv_to_ODE_laguerre}, we again argue by induction on $k$. The case $k=0$ is trivial: $\eta_{0;n}\equiv g_0\equiv 1$.
Assume that for some $k\in \N$ we already have, for $j=1,\ldots,k-1$,
\[
\sup_{s\in [0, An]\cap \Z}\left|\eta_{j;n}(s)-g_j(s/n)\right|\to 0
\quad\text{as }n\to\infty,
\]
for certain continuous functions $g_j:[0,\infty) \to  \C$.  From \eqref{eq:eta_system_difference_equations_laguerre} we can write, for $s\in [0, An]\cap \Z$,
\begin{equation}\label{eq:integral-form}
\eta_{k;n}(s)
=1+\frac{1}{n}\sum_{r=0}^{s-1}
\frac{1}{c_{k;n}^{ r+1}}
\Phi_k \left(c_{1;n}^{ r}\eta_{1;n}(r),\ldots,c_{k-1;n}^{ r}\eta_{k-1;n}(r); \alpha_n\right).
\end{equation}
Define a piecewise-constant function by
\begin{equation}\label{eq:g_k_n_u_def}
g_{k;n}(u)
:=
\frac{1}{c_{k;n}^{ \lfloor un\rfloor+1}}
\Phi_k \left(
c_{1;n}^{ \lfloor un\rfloor}\eta_{1;n}(\lfloor un\rfloor),\ldots,
c_{k-1;n}^{ \lfloor un\rfloor}\eta_{k-1;n}(\lfloor un\rfloor); \alpha_n\right),
\qquad
u\in [0, A],
\end{equation}
so that
\[
\eta_{k;n}(s)=1+\int_{0}^{s/n} g_{k;n}(u) \dd u \quad\text{for all } s\in [0, An] \cap \Z.
\]

\smallskip
\noindent
\emph{Pointwise limit of $g_{k;n}$.} Fix $u\in[0,A]$.
Then, using $c_{j;n}^{\lfloor un\rfloor}\to \eee^{j\alpha u}$ (since $\alpha=-\frac1{1+\beta}$) and the inductive hypothesis, for all $j\in \{1,\ldots, k-1\}$ we have
\[
\eta_{j;n}(\lfloor un\rfloor)\to g_j(u),\quad
c_{j;n}^{\lfloor un\rfloor}\eta_{j;n}(\lfloor un\rfloor)\to \eee^{j\alpha u}g_j(u),\quad
\frac{1}{c_{k;n}^{\lfloor un\rfloor+1}}\to \eee^{-k\alpha u},\quad
\alpha_n\to\alpha,
\]
as $n\to\infty$.
By continuity of the polynomial $\Phi_k$,
\[
\lim_{n\to\infty} g_{k;n}(u) =
G_k(u):=\eee^{-k\alpha u}
\Phi_k \left(\eee^{\alpha u}g_1(u),\ldots,\eee^{(k-1)\alpha u}g_{k-1}(u); \alpha\right).
\]
Next observe that $r^\ell \widehat B_\ell (x_1,x_2,\ldots, x_\ell) = \widehat B_\ell (r x_1,r^2 x_2,\ldots, r^\ell x_\ell)$ and thus, recalling~\eqref{eq:Phi_def_Bell_poly},
$$
G_k(u) = \eee^{-k\alpha u}
\Phi_k \left(\eee^{\alpha u}g_1(u),\ldots,\eee^{(k-1)\alpha u}g_{k-1}(u); \alpha\right)
=
\Phi_k \left(g_1(u),\ldots,g_{k-1}(u); \alpha\right).
$$

\smallskip
\noindent
\emph{Convergence of Riemann sums to integrals.}
By uniform boundedness of $\eta_{j;n}$ and of $c_{j;n}^{ s}$ (for $j\in \{1,\ldots, k-1\}$ and $s\in [0, An]\cap \Z$), the vector of arguments fed into $\Phi_k$ in~\eqref{eq:g_k_n_u_def} remains in a fixed compact set (independent of $n$ and $u$). Hence there exists $D_{k,A}<\infty$ such that $|g_{k;n}(u)|\le D_{k,A}$ for all $u\in[0,A]$ and all $n\in \N$.
By dominated convergence,
\[
\lim_{n\to\infty} \int_0^{A} |g_{k;n}(u)-G_k(u)| \dd u  =  0.
\]

\smallskip
\noindent
\emph{Completing the proof.} Define
\[
g_k(t):=1+\int_0^{t} G_k(u) \dd u,\qquad t\in[0,A].
\]
Then
\[
\sup_{ s\in [0, An] \cap \Z}\left|\eta_{k;n}(s)-g_k(s/n)\right|
=\sup_{0\le t\le A}\left|\int_0^{t} \left(g_{k;n}(u)-G_k(u)\right) \dd u\right| \le \int_0^{A} |g_{k;n}(u)-G_k(u)| \dd u \overset{}{\underset{n\to\infty}\longrightarrow} 0,
\]
which proves the desired uniform convergence for level $k$.
By construction, $g_k$ is $C^1$ with derivative $g_k'(t)=G_k(t) = \Phi_k (g_1(t),\ldots,g_{k-1}(t); \alpha)$ and $g_k(0)=1$,
which is exactly~\eqref{eq:fk-ode}.
\end{proof}

\subsection{Convergence of moments}
We are now ready to state a result on the convergence of moments for zero distributions of multiplicative Laguerre polynomials.
\begin{proposition}\label{prop:conv_moments_laguerre}
Let $(b_n)_{n\in \N} \subseteq \C$ and $(c_n)_{n\in \N} \subseteq \N_0$ be sequences such that
$$
\lim_{n\to\infty} \frac{b_n}{n} = \beta\in\C, \qquad \beta \neq -1,
\qquad
\lim_{n\to\infty} \frac{c_n}{n} = \gamma \geq 0.
$$
The dependence on the parameters $b_n$ and $\beta$ will be suppressed in the notation.  Recall that $\sigma_{k;n}(c_n)$ is the $k$-th moment of the zero distribution of $L_n^*(x; b_n, c_n)$. Then, for all $k\in \N_0$,
$$
\lim_{n\to\infty} \sigma_{k;n}(c_n) = f_k(\gamma) := \eee^{- \frac{\gamma k}{1+\beta}} g_k(\gamma),
$$
where the functions $g_k(\gamma)$, $k\in \N_0$, are the same as in Proposition~\ref{prop:conv_to_ODE_laguerre}, and the functions $f_k(\gamma)$, $k\in \N_0$, satisfy the following  system of ODE's:
\begin{equation}\label{eq:fk-ode-restated}
\begin{aligned}
&f_k'(\gamma)
=
-\frac 1 {1+\beta}  \sum_{j=1}^{k} j  f_j(\gamma)
\widehat{B}_{ k-j} \left(f_1(\gamma),\ldots, f_{k-j}(\gamma); - \frac 1 {1+\beta}\right),
\quad k\in \N,\ \gamma\geq 0,
\\
&f_{0}(\gamma) =  1 \quad (\text{for all } \gamma\geq 0), \qquad f_{k}(0) = 1 \quad (\text{for all } k\in  \mathbb N).
\end{aligned}
\end{equation}
\end{proposition}
\begin{proof}
Recall that $\sigma_{k;n}(c_n)=(c_{k;n})^{c_n} \eta_{k;n}(c_n)$ with $c_{k;n}=1-\frac{k}{b_n+n}$. Now,  $(c_{k;n})^{c_n} \to \eee^{-\gamma k/(1+\beta)} = \eee^{\alpha \gamma k}$, as $n\to\infty$,  where $\alpha = -1/(1+\beta)$. By Proposition~\ref{prop:conv_to_ODE_laguerre}, we have $\eta_{k;n}(c_n) \to g_k(\gamma)$ and  we conclude that $\sigma_{k;n}(c_n) \to \eee^{\alpha \gamma k} g_k(\gamma) = : f_k(\gamma)$. The triangular ODE system~\eqref{eq:fk-ode} for the functions $g_k$ turns into the following ODE system for the functions $f_k$:
\begin{align*}
f_k'(\gamma)
&=
(\eee^{\alpha \gamma k} g_k(\gamma))'
=
\alpha k f_k(\gamma) + \eee^{\alpha \gamma k} g_k'(\gamma)
=
\alpha k f_k(\gamma) + \alpha \eee^{\alpha \gamma k} \sum_{j=1}^{k-1} j  g_j(\gamma)
\widehat{B}_{ k-j} \left(g_1(\gamma),\ldots, g_{k-j}(\gamma); \alpha\right)
\\
&=
\alpha k f_k(\gamma) + \alpha \sum_{j=1}^{k-1} j  f_j(\gamma)
\widehat{B}_{ k-j} \left(f_1(\gamma),\ldots, f_{k-j}(\gamma); \alpha\right)
=
\alpha \sum_{j=1}^{k} j  f_j(\gamma)
\widehat{B}_{ k-j} \left(f_1(\gamma),\ldots, f_{k-j}(\gamma); \alpha\right),
\end{align*}
using the property $r^\ell \widehat B_\ell (x_1,x_2,\ldots, x_\ell) = \widehat B_\ell (r x_1,r^2 x_2,\ldots, r^\ell x_\ell)$ and $\widehat B_0 = 1$.
\end{proof}

\subsection{Free multiplicative Poisson distributions}\label{subsec:free_mult_poisson_properties}
Fix $\gamma\geq 0$ and $\beta\in \C$ such that either $\beta \in \R \backslash [-1,0]$ or $\Re \beta = -1/2$.
Consider a probability distribution $\nu_{\beta, \gamma}$ on $(0,\infty)$ (when $\beta \in \R \backslash [-1,0]$) or on the unit circle (when $\Re \beta = -1/2$) with the $S$-transform
$$
S(y)=\exp \left(\frac{\gamma}{\beta+1+y}\right).
$$
Existence of $\nu_{\beta, \gamma}$ has been established in~\cite[Lemmas~6.4, 7.2]{bercovici_voiculescu_levy_hincin}. Note that for $\gamma = 0$ we have $\nu_{\beta, 0}= \delta_1$. For $k\in \N_0$, let $m_k$ be the $k$-th moment of $\nu_{\beta, \gamma}$. In this section, our aim is to show that $m_k = f_k$, where $f_k$ are the functions appearing in Proposition~\ref{prop:conv_moments_laguerre}.  We shall also derive some other properties of $\nu_{\beta, \gamma}$, thus proving Proposition~\ref{prop:mult_free_poisson_properties}.  This will be done in several lemmas. Consider the generating function
$$
\Sigma(x):= \sum_{k=0}^{\infty} m_k  x^{-k}.
$$

\begin{lemma}
For complex $x$ with sufficiently large absolute value, the function $\Sigma(x)$ satisfies
\begin{equation}\label{eq:Sigma_laguerre_implicit_relation}
x=\frac{\Sigma(x)}{\Sigma(x)-1}\exp \left(-\frac{\gamma}{\beta+\Sigma(x)}\right).
\end{equation}
\end{lemma}
\begin{proof}
Let $\psi(z):=\sum_{k=1}^\infty m_k z^k$, so that \(\Sigma(x)=1+\psi(1/x)\) (with \(m_0=1\)).
By definition of the \(S\)-transform, Equation~\eqref{eq:S_transform_def}, the compositional inverse of $\psi$ satisfies
\begin{equation}\label{eq:chi_laguerre_proof}
\psi^{-1}(y)=\frac{y}{1+y} S(y) = \frac{y}{1+y}  \exp \left(\frac{\gamma}{\beta+1+y}\right).
\end{equation}
With \(y=\psi(1/x)=\Sigma(x)-1\), we obtain the implicit equation for \(\Sigma(x)\):
\begin{align*}
\frac{1}{x}
=\psi^{-1}\left(\Sigma(x)-1\right)
=\frac{\Sigma(x)-1}{\Sigma(x)}\exp \left(\frac{\gamma}{\beta+\Sigma(x)}\right).
\end{align*}
Taking the reciprocal of both sides completes the proof.
\end{proof}

\begin{lemma}
For all $k\in \N$, the $k$-th moment of $\nu_{\beta, \gamma}$ is given by
\begin{align}
m_k
&=\frac{1}{k} [t^{ k-1}] (1+t)^k
\exp \left(-\frac{k\gamma}{\beta+1+t}\right)
\label{eq:m_k_lahuerre_1}
\\
&=
\frac {\eee^{- k\gamma/(1 + \beta)}}k  \sum_{j=0}^{k-1}\binom{k}{j+1}\left(\frac {-1}{1+\beta}\right)^{j}
L_{j}^{(-1)}\left(\frac {k\gamma}{1+\beta}\right),
\label{eq:m_k_lahuerre_2}
\end{align}
where $L_{j}^{(-1)}(\cdot)$ are the associated Laguerre polynomials with $\alpha = -1$.
\end{lemma}
\begin{proof}
Taking $y=\psi(z)$ in~\eqref{eq:chi_laguerre_proof} gives
$
z= \frac{\psi(z)}{1+\psi(z)}  \exp (\frac{\gamma}{\beta+1+\psi(z)})
$,
which can be written as
$$
\psi(z) = z \Phi(\psi(z)),
\quad
\Phi(t):= (1+t)  \exp \left( - \frac{\gamma}{\beta+1+t}\right).
$$
Recall that $\psi(z) = \sum_{k=1}^\infty m_k z^k$. The Lagrange inversion formula (see~\cite[Equation~(2.1.1)]{gessel_lagrange_inversion}) gives $m_k = \frac{1}{k} [t^{k-1}] (\Phi(t))^k$, $k\in \N$, and proves~\eqref{eq:m_k_lahuerre_1}.

To prove~\eqref{eq:m_k_lahuerre_2}, we need the generating function $\eee^{-\frac{c}{1-u}}
=\eee^{-c}\eee^{-\frac{c u}{1-u}}
=\eee^{-c}\sum_{j=0}^\infty L_j^{(-1)}(c) u^j$.
With $c= \frac{k\gamma}{1+\beta}$ and $u = -\frac{t}{1+\beta}$, this gives
\begin{equation}\label{eq:laguerre_expansion_exponential}
\eee^{-\frac{k\gamma}{1+\beta+t}}
=
\eee^{-\frac{k\gamma}{1+\beta}\cdot\frac{1}{1+t/(\beta+1)}}
=
\eee^{-\frac{k\gamma}{1+\beta}} \sum_{j=0}^\infty L_j^{(-1)}\left(\frac{k\gamma}{1+\beta}\right) \left(\frac{-1}{1+\beta}\right)^j t^j.
\end{equation}
Multiplying this expansion with  $(1+t)^k=\sum_{j=-1}^{k-1}\binom{k}{j+1} t^{k-1-j}$ and extracting  the coefficient of $t^{k-1}$ proves~\eqref{eq:m_k_lahuerre_2}.
\end{proof}

\begin{lemma}
For every $k\geq 2$, the $k$-th free cumulant of $\nu_{\beta, \gamma}$ is  given by
\begin{align}\label{eq:kappa_k_free_mult_poisson_restate}
\kappa_k
=
\frac{1}{k} [t^{ k-1}] \eee^{-k \gamma/(1+\beta+t)}
=
\frac {\eee^{- k\gamma/(1 + \beta)}}k \left(\frac {-1}{1+\beta}\right)^{k-1}
L_{k-1}^{(-1)}\left(\frac {k\gamma}{1+\beta}\right).
\end{align}
\end{lemma}
\begin{proof}
Recall the basic identity $S(y) R (yS(y))=1$ and set $\phi(y):= 1/S(y)=\exp (-\frac{\gamma}{1+\beta+y})$.
Then $R(u)=\phi (y(u))$, where
\[
u := y S(y)=\frac{y}{\phi(y)}.
\]
Note that $\phi(0)\neq 0$.  The Lagrange--Bürmann formula (see~\cite[Equation~(2.1.1)]{gessel_lagrange_inversion}) states that for every analytic function $G$, we have  $[u^{r}] G (y(u))=\frac{1}{r} [t^{ r-1}] G'(t) \phi(t)^{r}$,  for all $r\in \N$.  Taking $G=\phi$ gives
\[
\kappa_k = [u^{k-1}]R(u) =  [u^{k-1}] G \left(y(u)\right)=\frac{1}{k-1} [t^{ k-2}] \phi'(t) \phi(t)^{k-1}
=
\frac{1}{k(k-1)} [t^{ k-2}] (\phi(t)^{k})'
=
\frac{1}{k} [t^{ k-1}] \phi(t)^{k},
\]
for all $k\geq 2$. It is easy to check that $\kappa_1 = 1$.  This proves the first formula in~\eqref{eq:kappa_k_free_mult_poisson_restate}.
The second formula follows from~\eqref{eq:laguerre_expansion_exponential}.
\end{proof}
From now on, we write $\Sigma(x)= \Sigma(x;\gamma)$ and $m_k = m_k(\gamma)$, suppressing the dependence on the parameter $\beta$ which we keep fixed.
\begin{lemma}\label{lem:Sigma_PDE_laguerre}
The function $\Sigma(x;\gamma)$ satisfies the PDE
\begin{equation}\label{eq:PDE_for_Sigma_proof_laguerre}
\partial_\gamma \Sigma(x;\gamma)  =  \frac{x \partial_x \Sigma(x;\gamma)}{\beta+\Sigma(x;\gamma)},
\qquad
\Sigma(x;0)=\sum_{k=0}^\infty x^{-k}=\frac{x}{x-1}.
\end{equation}
\end{lemma}
\begin{proof}
Taking the logarithm of~\eqref{eq:Sigma_laguerre_implicit_relation} gives $G(\Sigma(x;\gamma),x,\gamma) = 0$, where
\[
G(\Sigma,x,\gamma):=\log x - \log\Sigma + \log(\Sigma-1) + \frac{\gamma}{\beta+\Sigma}.
\]
Applying the operators $\partial_\gamma$ and $\partial_x$ to the identity $G(\Sigma(x;\gamma),x,\gamma) = 0$ gives
\begin{align*}
&\partial_\Sigma G(\Sigma(x;\gamma),x,\gamma) \partial_\gamma \Sigma(x;\gamma) + \partial_\gamma G(\Sigma(x;\gamma),x,\gamma) = 0,
\\
&\partial_\Sigma G(\Sigma(x;\gamma),x,\gamma) \partial_x \Sigma(x;\gamma) + \partial_x G(\Sigma(x;\gamma),x,\gamma) = 0.
\end{align*}
It follows that $\partial_\gamma \Sigma = \frac{\partial_\gamma G}{\partial_x G} \partial_x \Sigma$.  Compute
$\partial_\gamma G=\frac{1}{\beta+\Sigma}$ and $\partial_x G=\frac{1}{x}$.  This gives $\partial_\gamma \Sigma(x;\gamma)  =  \frac{x \partial_x \Sigma(x;\gamma)}{\beta+\Sigma(x;\gamma)}$.
\end{proof}

\begin{lemma}\label{lem:m_k_gamma_ODE_laguerre}
Recall that $m_k(\gamma)$ is the $k$-th moment of $\nu_{\beta, \gamma}$. Then,
\begin{equation}
\begin{aligned}
&m_k'(\gamma)
=
 -  \frac 1 {1 +  \beta}
\sum_{j=1}^{k} j m_j(\gamma)
\widehat B_{k-j} \left(m_1(\gamma),\ldots, m_{k-j}(\gamma); -\frac 1 {1 + \beta}\right),
\quad k\in \N,\ \gamma\geq 0,
\\
&m_{0}(\gamma) =  1 \quad (\text{for all } \gamma\geq 0), \qquad m_{k}(0) = 1 \quad (\text{for all } k\in  \mathbb N).
\end{aligned}
\end{equation}
\end{lemma}
\begin{proof}
The proof is by coefficient expansion in the PDE~\eqref{eq:PDE_for_Sigma_proof_laguerre}.  Recall that $\Sigma(x;\gamma)=\sum_{k=0}^\infty m_k(\gamma) x^{-k}$. Taking derivative in $x$ gives
\begin{equation}\label{eq:proof_m_k_gamma_laguerre_1}
x \partial_x\Sigma(x;\gamma) = - \sum_{j=1}^\infty j m_j(\gamma) x^{-j}.
\end{equation}
On the other hand, expanding into geometric series gives
$$
\frac{1}{\beta+\Sigma(x;\gamma)}
=\frac{1}{\beta+1+(\Sigma(x;\gamma) - 1)}
=
\frac 1 {1+\beta} \sum_{p=0}^\infty \left(-\frac{1}{1+\beta}\right)^p
\left(\Sigma(x;\gamma) - 1\right)^{p}.
$$
For coefficient extraction we use the  definition of  partial Bell polynomials
\[
\left(\Sigma(x;\gamma) - 1\right)^{p}
=
\left(\sum_{k=1}^\infty m_k(\gamma) x^{-k}\right)^{ p}
=
\sum_{r=p}^\infty \widehat B_{r,p}\left(m_1(\gamma),\ldots, m_{r-p+1}(\gamma)\right) x^{ - r}.
\]
Interchanging the sums and using the definition of weighted ordinary Bell polynomials, we obtain
\begin{align}
\frac{1}{\beta+\Sigma(x;\gamma)}
&=
\frac 1 {1+\beta} \sum_{r=0}^\infty \sum_{p=0}^r \left(-\frac{1}{1+\beta}\right)^p
 \widehat B_{r,p}\left(m_1(\gamma),\ldots, m_{r-p+1}(\gamma)\right) x^{ - r}  \notag
 \\
&=
\frac 1 {1+\beta} \sum_{r=0}^\infty x^{ - r}
\widehat B_{r}\left(m_1(\gamma),\ldots, m_{r}(\gamma); -\frac{1}{1+\beta}\right).  \label{eq:proof_m_k_gamma_laguerre_2}
\end{align}
Taking the product of~\eqref{eq:proof_m_k_gamma_laguerre_1} and~\eqref{eq:proof_m_k_gamma_laguerre_2} and extracting the coefficient of $x^{-k}$ gives
$$
[x^{-k}]\frac{x \partial_x \Sigma(x;\gamma)}{\beta+\Sigma(x;\gamma)}
=
-\frac 1 {1+\beta} \sum_{j=1}^k j m_j(\gamma) \widehat B_{k-j}\left(m_1(\gamma),\ldots, m_{k-j}(\gamma); -\frac{1}{1+\beta}\right).
$$
To complete the proof, recall from~\eqref{lem:Sigma_PDE_laguerre} that $m_k'(\gamma) = [x^{-k}]\partial_\gamma \Sigma(x;\gamma) = [x^{-k}]\frac{x \partial_x \Sigma(x;\gamma)}{\beta+\Sigma(x;\gamma)}$.
\end{proof}

\begin{lemma}
For all admissible $\beta$ and $\gamma$, the $k$-th moment of $\nu_{\beta, \gamma}$ equals $f_k(\gamma)$, where $f_k(\gamma)$ are the functions appearing in Proposition~\ref{prop:conv_moments_laguerre}.
\end{lemma}
\begin{proof}
Indeed, the ODE system satisfied by $m_k(\gamma)$, see Lemma~\ref{lem:m_k_gamma_ODE_laguerre}, is identical to that satisfied by $f_k(\gamma)$, see Proposition~\ref{prop:conv_moments_laguerre}. Uniqueness of solution gives the claim.
\end{proof}

\subsection{Weak convergence}
We are now ready to complete the proof of Theorem~\ref{theo:multiplicative_laguerre_polys_zeros}.
In the previous sections we showed that the $k$-th moment of $\llbracket  L_n^*(\cdot; b_n,c_n)\rrbracket_n$ converges to the $k$-th moment of a probability measure $\nu_{\beta, \gamma}$ with the $S$-transform $S(z) = \exp (\gamma/ (\beta+1+y))$. In the positive case, this implies weak convergence since $\nu_{\beta, \gamma}$ is compactly supported. In the unitary case, convergence of moments of any order $k\in \N$ (and hence, $k\in \Z$, by complex conjugation), implies weak convergence. The proof of Theorem~\ref{theo:multiplicative_laguerre_polys_zeros} is complete.

%Not necessary: To conclude weak convergence in  the unitary case, we need convergence of moments of any order $k\in \Z$  (rather than $k\in \N$). In the unitary case,  $-n-b_n = \overline{b_n}$ and $-1-\beta = \overline \beta$.  The property $L_n^*(x; -n-b_n, c_n)=(-1)^{ n+c_n} x^{n} L_n^*(x^{-1}; b_n, c_n)$  (which follows from the definition of $L_n^*$) implies that the measure $\llbracket  L_n^*(\cdot; b_n,c_n)\rrbracket_n$ is the image of $\llbracket  L_n^*(\cdot; -n-b_n,c_n)\rrbracket_n$ under the map $z\mapsto \bar z$; the same is true for the measures $\nu_{\beta, \gamma}$ and $\nu_{-1-\beta, \gamma}$ by uniqueness of the $S$-transform. Hence, convergence of $k$-th moments implies convergence of moments of order $-k$. By Weyl's criterion, we conclude that  $\llbracket  L_n^*(\cdot; b_n,c_n)\rrbracket_n$ converges weakly to $\nu_{\beta, \gamma}$. The proof of Theorem~\ref{theo:multiplicative_laguerre_polys_zeros} is complete.

\section*{Acknowledgement}
This work would hardly have been possible without the patient support and creative ideas provided by ChatGPT.  Supported by the German Research Foundation under Germany's Excellence Strategy  EXC 2044/2 -- 390685587, Mathematics M\"unster: Dynamics - Geometry - Structure and by the DFG priority program SPP 2265 Random Geometric Systems.

\bibliographystyle{amsplain}
\bibliography{bib_mult_hermite}

\end{document}